\documentclass[12pt,a4paper]{article}%
\usepackage{amsmath}
\usepackage{graphicx}
\usepackage{epsfig}%
\usepackage{amsfonts}%
\usepackage{amssymb}

\usepackage{color}

\usepackage[normalem]{ulem}

\newtheorem{theorem}{Theorem}[section]

\newtheorem{definition}[theorem]{Definition}

\newtheorem{lemma}[theorem]{Lemma}
\newtheorem{proposition}[theorem]{Proposition}
\newtheorem{remark}[theorem]{Remark}

\newenvironment{proof}[1][Proof]{\noindent \emph{#1.} }{\hfill \ 
\rule{0.5em}{0.5em}}

\makeatletter\@addtoreset{equation}{section}\makeatother
\makeatletter\@addtoreset{figure}{section}\makeatother
\makeatletter\@addtoreset{table}{section}\makeatother
\begin{document}

\title{Fast tensor-based electrostatic energy  calculations  \\
in the perspective of protein-ligand  docking problem }

\author{Peter Benner \thanks{Max Planck Institute for Dynamics of Complex 
Technical Systems, Sandtorstr.~1, D-39106 Magdeburg, Germany
({\tt benner@mpi-magdeburg.mpg.de})}
\and Boris N. Khoromskij \thanks{Max Planck Institute for
        Mathematics in the Sciences, Inselstr.~22-26, D-04103 Leipzig,
        Germany ({\tt bokh@mis.mpg.de}); 
        Max Planck Institute for Dynamics of Complex 
Technical Systems, Sandtorstr.~1, D-39106 Magdeburg, Germany;
        OVGU University of Magdeburg, Germany.}
\and
Venera Khoromskaia \thanks{Max Planck Institute for
        Mathematics in the Sciences,  Inselstr.~22-26, D-04103 Leipzig, Germany ({\tt vekh@mis.mpg.de})}
  \and Matthias Stein \thanks{Max Planck Institute for Dynamics of Complex 
Technical Systems, Sandtorstr.~1, D-39106 Magdeburg, Germany
        ({\tt matthias.stein@mpi-magdeburg.mpg.de})}
        \\  
            }


\maketitle

\begin{abstract}
We propose and justify a new  approach  for fast calculation of 
 the electrostatic interaction energy of clusters of charged particles   
 in constrained energy minimization 
 in the framework of rigid  protein-ligand docking.
  Our ``blind search'' docking technique is based on  the   low-rank  
 range-separated (RS) tensor-based representation of the free-space electrostatic potential of the biomolecule  represented on large   $n\times n\times n$ 3D grid. 
  We show that both the collective electrostatic potential of a complex protein-ligand system
    and the respective electrostatic interaction  energy can be calculated  by 
 tensor techniques in $O(n)$-complexity, such that the numerical cost for energy calculation 
 only mildly (logarithmically) depends on the number of particles in the system.  
 Moreover, tensor representation of the 
 electrostatic potential enables usage of large 3D   Cartesian grids 
 (of the order of $n^3 \sim 10^{12}$), which could allow the accurate modeling of complexes with several large proteins.
In our approach selection of the correct geometric pose predictions in the localized posing process is based on the control 
of  van der Waals distance  between the target molecular clusters. 
Here, we confine ourselves by  constrained  minimization of the energy functional by using only fast tensor-based  free-space electrostatic energy recalculation for various rotations and translations of both clusters.  
Numerical tests of the electrostatic energy-based ``protein-ligand  docking'' algorithm applied to synthetic and realistic input data present a proof of concept for rather complex particle configurations. 
  The method may be used in the framework of the traditional 
stochastic or deterministic posing/docking techniques. 
 \end{abstract}



\section{Introduction}\label{sec:Intro}

Docking of large complex molecular clusters is one of the challenging computational problems
 in biomolecular computer simulation and drug design \cite{Nature_Rev:2004,Collony83}. 
 \emph{In silico} techniques are now   
 important procedures   in the synthesis of new structures both in protein formation and in 
 construction of nano-structures compounds.   The many-particle systems of interest
may arise from the modeling, design or classification of large biomolecules or 
shape optimization of the bulk materials modeled via lattice structured systems 
\cite{HoEa:88,DavMcC:90,SharpHonig90,Honig95,Neves-Petersen2003}.
In this concern both the energy functional and the many-body electrostatic potential  play the central role in the quantitative approaches.
The most popular implicit solvent model for computation of the electrostatic potential in bio molecules is 
  the Poisson-Boltzmann equation (PBE) in $\mathbb{R}^3$, 
which is traditionally solved by finite element or finite difference methods, 
see \cite{Holst94,WaBa:04,Baker:04,BeKKKS:21,KKKSBen:21}. 
The domain decomposition approach \cite{LiStCaMaMe:13} could be adapted in parallel computations.

The problem of protein-ligand docking is being successfully attacked  since decades \cite{Halperin:2002,charmm:98,McCam:2001,Nat:77}
and there is a 
large number of software packages for the numerical treatment of this task  \cite{GROM:99,Amber:2005,Amber:2013,Charm:2003}
 by using either the stochastic molecular dynamics \cite{Noe:2015,DMVR:2015,Noe:2017}  or by deterministic methods. 
  The latter are based on using  a large variety of scoring 
 functions and heuristic solution strategies\footnote{Commonly used methods for solution of the docking problem are based on the use of some a priori known localizations of docking positions.}, see 
 \cite{Meng:92,Nature_Rev:2004,HGZ:2010,LiuWang:15,HuZou:2010,DMVR:2015,Peng:2018}. 
 Reviews on different type of scoring  functions  
 for the protein docking problem  are presented in
 \cite{Nature_Rev:2004,HGZ:2010,Tsujikawa:2016,LiBa:2019}. 
  Also, the traditional methods may employ the  similarity search  \cite{MIFs_Wade:2001}, 
 geometric complementarity \cite{Wade:96,ChWe:2003} 
 and/or selection of the possible conformations by using  large data sets of descriptions, 
 the so-called virtual screening \cite{Nature_Rev:2004}, as well as 
 binding energy analysis \cite{Wade:95,Stein:2004}.
 Usage of the  molecular interaction fields was  considered in \cite{MIFs_Wade:2001}.
    Recent packages for protein structure prediction are based on 
 deep learning techniques via neural networks 
 \cite{AlphF:2021,DeepL:2020}.

In view of huge achievements in the area of protein-ligand  docking simulation
 and a large number of the corresponding  software one can notice that there are still some bottleneck problems,
  requiring novel mathematically justified approaches.
  The efficient calculations of many-particle electrostatic potential and interaction energy are of great importance in the bio-molecular modeling, while
 the traditional methods of their numerical treatment  remain cumbersome and 
 have severe size limitations \cite{Warshel:2006,LuChen:2006}. 

     Our goal here is to attract the attention of the community to  
  benefits of the recent  tensor-based techniques 
  (substantiated  on a solid mathematical background)
  for fast calculation of the electrostatic interaction 
  energy and force field of bio-molecules
 which could be  promising and may lead to significant enhancement of existing numerical schemes.  We aim to present the proof of concept for the new efficient 
 method of electrostatic calculations for energy optimization in many-particle systems.

 In the recent years a noticeable progress has been attained in the numerical modeling of the  collective electrostatic  
 potential in many-particle systems by using the tensor decomposition of the generating Coulomb potential (Newton kernel) 
 \cite{HaKhtens:04I,Khor-book_2018,BKK_RS:18}.  
 In the case of finite lattice-structured systems
 the tensor summation method 
 \cite{VeBoKh_CPC:14,Khor-book_2018}
 essentially outperforms the traditional Ewald-type methods \cite{Ewald:27,Hune_Ewald:99}, since   
 it allows to compute  not only the  interaction  energy, but also the real-space 3D collective  
 electrostatic potential of large $N$-body molecular systems in the whole 
 computational box in the form of its data-sparse parametrization in tensor format. 
  
Rank structured range-separated (RS) tensor format for calculation of the free-space collective 
electrostatic potential for
many-particle systems of general type have been introduced and analyzed in \cite{BKK_RS:18}, where  
it was proven that the rank parameter of the  long-range part of the  potential  
only weakly (logarithmically) depends on the number of particles  in a  molecular system as well as on the approximation accuracy. 
 In this case, the free-space \emph{collective electrostatic potential of large atomic clusters} 
 is  computed by rather simple tensor summations 
 (using shifts and summations of vectors) without the need to solve the elliptic PDEs, as for example, for PBE.  
 Moreover, the RS tensor representation of the electrostatic potential  provides a beneficial 
 scheme for computation of the \emph{electrostatic interaction energy} 
 of a cluster of $N$ charged particles, with the cost that only 
 mildly (logarithmically) depends on the number of particles, $N$.
 The force field can be then easily computed.
  Recently, RS tensor format has been successfully applied to the   tensor decomposition of Dirac delta and elliptic operator inverse \cite{BoKH_Dirac:20} which was the 
background for the new  
regularization scheme of the Poisson--Boltzmann type equations with singular source terms 
\cite{BeKKKS:21,KKKSBen:21}, see also \cite{KwFSBen:22} for an application of reduced basis method in the framework of the new regularization scheme.

 We recall that the tensor numerical methods provide essentially one-dimensional 
 $O(n)$ computational complexity 
 for the three-dimensional problems discretized on $n\times n \times n$ 3D Cartesian 
 grids \cite{Khor-book_2018,Hackbusch:12}, opposite to
  the volume complexity of the order of $O(n^3 \log(n))$ 
 for traditional asymptotically optimal numerical approaches. 
  Tensor numerical methods already proved to be efficient 
 in computational quantum chemistry \cite{khor-ml-2009,Khor_bookQC_2018}, in optimal control problems \cite{Heid:21}, 
 for solution of the 3D elliptic equations in 
 stochastic homogenization \cite{BoVeKh:23}
 and in numerical modeling of many  other real life problems
 \cite{Khor_bookQC_2018,FokPla:12,KhQuant:09,Osel_TT:11,Kilmer:22}.
  Resent hybrid tensor approximation methods based on the combination of Chebyshev polynomial interpolation and Tucker tensor decomposition \cite{BeKhKhBS1:25} paves the way to further enhancement of tensor approximation techniques, see also  \cite{Kress:21}.
 The main focus of tensor calculus is maintaining the low tensor rank; in presented approach the robust control of
 tensor rank is tackled  by   the robust 
  rank-truncation methods using the reduced higher order singular value 
  decomposition (RHOSVD) \cite{khor-ml-2009}.

  In this paper, we present the new approach for computing the electrostatic
  interaction energy in  the framework of rigid-body docking problem for multi-particle systems  via constrained energy minimization.
  It is based on using the RS tensor techniques of $O(n)$ complexity for representation
   of the collective electrostatic potentials on $n\times n\times n $ 
  3D Cartesian grid. This paves the way to fast interaction cross-energy calculation
  for large bio-molecules clusters (electrostatic part of the protein-ligand binding energy) at the cost 
  that is almost independent on the size of molecule, $N$.
  Tensor approach enables computation of the electrostatic potential in the canonical tensor format 
  represented on the large 3D grids of the order of $n^3 \approx 10^{12}$. 
  This property was already gainfully utilized in high-accuracy electronic structure 
 calculations\footnote{High accuracy tensor-based solver for the Hartree-Fock equation
 includes calculation of the nuclear potential energy operator, where the 
 tensor-based representation of the nuclear charges is applied.
 It uses as well  the tensor-based 3D Coulomb potential  
  for the accurate calculation of the 3D convolution operators  with the 
  Newton kernel using grids of size up to $n^3 \approx 10^{14}$ (without high-performance computing).}. 
  Large grid sizes enable fast and accurate electrostatic calculations for large protein complexes.

  Thus, we benefit from the fast grid-based tensor evaluation 
  of the long-range part in the collective electrostatic potential of complex molecular systems. 
  In our numerical experiments we try to find the 
  conforming sites for the ligand by using  only 
  the electrostatic part of the binding energy in estimation for the minimum positions on the potential energy surface.    
  Preventing the possible overlap of protein-ligand clusters in the posing/docking process is controlled  by estimating 
  the van der Waals distance. 
  The acceptable posing positions of ligand could be determined  by the direct local energy optimization 
  in the enveloping strip around the protein combined, if required, with the gradient search based on calculation of the 
  collective force vector for rigid ligand.

  Thus, we summarize the advantages of the RS tensor approach:
    \begin{itemize}
  \item The low-rank parametric representation in the RS tensor format
  of the long-range part in collective electrostatic potential of the  bio-molecule  
  represented on large 3D Cartesian grids of size $n^3 \sim 10^{12}$ (using Matlab on a laptop).
    All tensor operations are performed in  $O(n)$ complexity.
  \item Super-fast electrostatic interaction energy calculations
    in \emph{logarithmic} complexity w.r.t. the number of particles  in the protein, $N$.
   \item An opportunity for fast evaluation of the collective force field for the rigid ligand cluster by using grid-based 
   tensor representation of the multi-particle electrostatic potential 
   (and the related energy).
   \end{itemize}

   We consider two numerical examples, first  with a pair of flat synthetic 
  clusters in a volume, and an example with a small protein containing 379 atoms.
  In the latter example, we extract and separate a small fragment of the protein (a group of connected particles)
  and then, our algorithms make a ``blind search'' of
  the initial location of this fragment (``ligand'')
  using only the electrostatic interaction calculations.
  
  This  ``blind search docking'' process is controlled only by evaluating the electrostatic interaction energy  
  which is computed by using the canonical tensor representation of the 
  electrostatic potential at the position of every particle in traveling ligand.  
  The samples on multidimensional potential energy surface (PES) are computed for azimuthal  
  rotations and translations of the ligand with respect  to the protein. 
  At every step on the ligand trajectory the van der Waals distance from ligand to the protein 
  surface is controlled, thus giving the possibility to trace the cavities and the lumps on the protein surface.
  
  Then, according to the determined minimum(s) on the   rotational/translational  
  PES, we recover the respective admissible docking positions. 
  However, one can omit some unnecessary rotation positions for 
  the ligand, by using a kind of ``learning'' procedure for the potential of the protein at
  a given site. Thus, for example, one can noticeably reduce the combinatorial expense of the 
  deterministic  calculations by selecting only most preferable  
   rotations and translations.

   In numerical techniques based on deterministic/stochastic dynamics the efficient gradient 
(or stochastic gradient) calculations play the central role.
In the framework of docking problems one deals with the dynamics of \emph{rigid clusters of charged particles}.
In this setting the force calculation requires certain modification, i.e. the only collective force vector of the ligand, compared with the case 
of single particles, see \S\ref{ssec:Forces_Applic}. We discuss fast tensor-based 
force calculation for the rigid cluster of charged particles (say, ligand) in Section \ref{ssec:Forces-Cluster}.

The rest of the paper is structured as follows. Section 2 provides the short introduction 
to RS tensor format. In Section 3 we outline the problem setting in the form of constrained 
energy functional minimization 
in multi-parametric configuration space. We propose the special representation of binding energy for cross protein-ligand interactions that allows
the direct use of the low-rank tensor decomposition for electrostatic potential of protein. 
We then present tensor-based method for 
fast collective force calculation and describe several algorithmic schemes to enhance the energy evaluation. 
The rigorous mathematical arguments concerning the numerical complexity of considered computational schemes
are presented.
We also discuss the main details on the energy minimization algorithms for protein-ligand systems 
over multi-parametric PES by using RS tensor format and using the canonical/Tucker low-rank decompositions of multidimensional data.
Numerical tests demonstrate the main features of the presented docking techniques.
Section 4 outlines further prospects for the enhancement and application of the introduced tensor methods.

 \section{Brief introduction to the RS tensor format}
 \label{sec:ten_formats}
 
 First, we recall the grid-based method for the low-rank canonical
representation of a spherically symmetric kernel function $p(\|x\|)$,
$x\in \mathbb{R}^d$ for $d=2,3,\ldots$, by its projection onto the set
of piecewise constant basis functions, see \cite{BeHaKh:08,Khor-book_2018,Khor_bookQC_2018} for the case of
the Newton kernel $p(\|x\|)=\frac{1}{\|x\|}$,  for $x\in \mathbb{R}^3$.
A single reference potential like $1/\|x\|$ can be represented on a fine 3D 
$n\times n\times n$ Cartesian grid as a low-rank canonical tensor \cite{HaKhtens:04I,BeHaKh:08}.

In the computational box  $\Omega=[-b,b]^3$, let us introduce the uniform $n \times n \times n$ 
rectangular Cartesian grid $\Omega_{n}$ with mesh size $h=2b/n$ ($n$ even).
Let $\{ \psi_\textbf{i}\}$ be a set of tensor-product piecewise constant basis functions,
$  \psi_\textbf{i}(\textbf{x})=\prod_{\ell=1}^3 \psi_{i_\ell}^{(\ell)}(x_\ell)$,
for the $3$-tuple index ${\bf i}=(i_1,i_2,i_3)$, $i_\ell \in I_\ell=\{1,...,n\}$, $\ell=1,\, 2,\, 3 $.
The generating radial kernel function $p(\|x\|)$ is discretized by its projection onto the basis
set $\{ \psi_\textbf{i}\}$
in the form of a third order tensor of size $n\times n \times n$, defined entry-wise as
\begin{equation}  \label{galten}
\mathbf{P}:=[p_{\bf i}] \in \mathbb{R}^{n\times n \times n}, \quad
 p_{\bf i} =
\int_{\mathbb{R}^3} \psi_{\bf i} ({x}) p(\|{x}\|) \,\, \mathrm{d}{x}.
\end{equation}

Using the Laplace-Gauss transform for analytic function $p(z)$, 
and then applying the sinc-quadrature approximation to the corresponding integral representation
on the real axis,
we obtain the approximation to the radial function $p(z)= p(\|{x}\|)$
by a sum of Gaussians \cite{Braess:95,Stenger,HaKhtens:04I},
\begin{equation}  \label{eqn:sinc}
 p(z)=  \frac{2}{\sqrt{\pi}}\int_{\mathbb{R}_+} \tilde{p}(t) e^{- t^2 z^2 } dt  
  \approx 
\sum_{k=-K}^{K} a_k e^{- t_k^2\|{x}\|^2}= 
\sum_{k=-K}^{K} a_k  \prod_{\ell=1}^d e^{-t_k^2 x_\ell^2},
 \end{equation}
  where $a_k=\frac{2}{\sqrt{\pi}}\tilde{p}(t_k)$ with the sinc quadrature points $t_k$.
 Here each Gaussian term is a separable function.
For certain class of analytic functions $p(z)$  
the exponentially fast in $K$ convergence on the real axis $z=\|x\|$ can be proven,
\[
0< h \leq \|{x}\|: \quad
\left|p(\|{x}\|) - \sum_{k=-K}^{K} a_k \prod_{\ell=1}^d e^{-t_k^2 x_\ell^2}\right|  
\le \frac{C}{h}\, \displaystyle{e}^{-\beta \sqrt{K}},  
\quad \text{with} \, C,\beta >0,
\]
for the proper choice of the quadrature points $t_k\in \mathbb{R}$, and weights $a_k >0$.
In particular, this applies to the case $p(z)=\frac{1}{z}, z=\|x\|$, where the Gaussian integral representation 
takes a form 
$ 1/\|{x}\| =  \frac{2}{\sqrt{\pi}}\int_{\mathbb{R}_+}  e^{- \|x\|^2 t^2  } dt, \,\, \|x\|>0 $.

As result, the $3$rd order tensor $\mathbf{P}$ can be approximated by
the $R$-term   canonical representation
(see \cite{HaKhtens:04I,BeHaKh:08,VeBoKh_CPC:14} for details),
\begin{equation} \label{eqn:sinc_general}
    \mathbf{P} \approx  \mathbf{P}_R =
  \sum\limits_{k=1}^{R} {\bf p}^{(1)}_k \otimes {\bf p}^{(2)}_k \otimes {\bf p}^{(3)}_k
\in \mathbb{R}^{n\times n \times n}, 
\end{equation}
where ${\bf p}^{(\ell)}_k \in \mathbb{R}^n$ and the rank parameter $R\leq 2K+1$ scales logarithmically in 
the accuracy $\varepsilon>0$, $R=O(|\log \varepsilon |^2)$.
Thus, a spherically symmetric kernel function $p(\|x\|)= \frac{1}{\|x\|}$, $x\in \mathbb{R}^d$ for $d=3$,   
representing the Newton potential, is parametrized  on the grid by the canonical 
tensor $\mathbf{P}_R$ with a certain accuracy $\varepsilon$. 
This representation have been used in many applications of tensor numerical methods  in computational 
quantum chemistry \cite{Khor_bookQC_2018}.

 The range separated (RS) tensor format     introduced in 
 \cite{BKK_RS:18} is well suited for modeling of
 the long-range interaction potentials of multi-particle systems of general type. 
  It is based on the partitioning of the reference tensor representation for the Newton kernel  
  into long- and short-range parts.

According to the canonical rank-$R$ tensor representation of the Newton kernel (\ref{eqn:sinc_general}) 
as a sum of Gaussians, 
and following the construction of the range-separated 
decomposition of the reference kernel \cite{BKK_RS:18}, 
one can recognize the short- and long-range parts in $\mathbf{P}_R$,
by distinguishing their effective supports  
\[
 \mathbf{P}_R = \mathbf{P}_{R_s} + \mathbf{P}_{R_l},
\]
with
\begin{equation} \label{eqn:Split_Tens}
    \mathbf{P}_{R_s} =
\sum\limits_{k=1}^{R_s} {\bf p}^{(1)}_k \otimes {\bf p}^{(2)}_k \otimes {\bf p}^{(3)}_k, 
\quad \mathbf{P}_{R_l} =
\sum\limits_{k=R_s+1}^R {\bf p}^{(1)}_k \otimes {\bf p}^{(2)}_k \otimes {\bf p}^{(3)}_k,
\end{equation}
where $R_s$ and $R_l$ denote the canonical ranks of the short- and long-range components in the total 
potential $\mathbf{P}_R$, respectively.
Figure \ref{fig:Vec_Newton} shows the form  of Gaussians in the long-range and the short-range parts of the tensor representation of the Newton kernel 
in x-axis.
\begin{figure}[tbh]
\centering
\includegraphics[width=7.0cm]{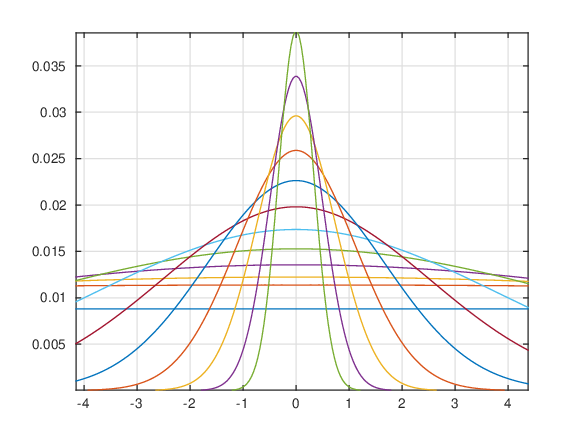}
\includegraphics[width=7.1cm]{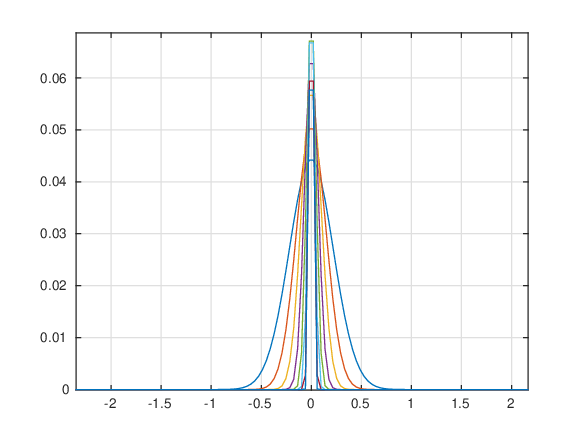}
\caption{\small The long-range (left) and the short-range (right) parts of the tensor representation of the Newton kernel in x-axis.}
\label{fig:Vec_Newton}
\end{figure}

Given a reference tensor 
$\widetilde{\bf P}_R\in \mathbb{R}^{2n \times 2n \times 2n}$, 
defined on a twice larger grid, 
then the collective electrostatic potential ${\bf P}_0\in \mathbb{R}^{n \times n \times n}$ is computed as a
  a weighted sum of the single Newton kernels $\widetilde{\bf P}_R$ placed at the points 
  of the coordinates of $N$ charged particles, $x_\nu$, $\nu=1,\ldots ,N$.  
Specifically, it is constructed by a direct sum of shift-and-windowing transforms of the reference 
tensor $\widetilde{\bf P}_R$ (see \cite{VeBoKh_CPC:14} for more details),
\begin{equation}\label{eqn:Total_Sum}
 {\bf P}_0 = \sum_{\nu=1}^{N} {z_\nu}\, {\cal W}_\nu (\widetilde{\bf P}_R)=
 \sum_{\nu=1}^{N} {z_\nu} \, {\cal W}_\nu (\widetilde{\mathbf{P}}_{R_s} + \widetilde{\mathbf{P}}_{R_l})
 =: {\bf P}_s + {\bf P}_l.
\end{equation}
The shift-and-windowing transform ${\cal W}_\nu$ maps a reference tensor 
$\widetilde{\bf P}_R\in \mathbb{R}^{2n \times 2n \times 2n}$ onto its sub-tensor 
of smaller size $n \times n \times n$, obtained by first shifting the center of
the reference tensor $\widetilde{\bf P}_R$ to the grid-point $x_\nu$ and then restricting 
(windowing) the result onto the computational grid $\Omega_n$.
  
The difficulty of the tensor representation   (\ref{eqn:Total_Sum})  is
  that the number of terms in the canonical representation of the full tensor sum ${\bf P}_0$
increases merely proportionally to the number $N$ of particles in the system with the multiple $R_l$.

This problem is solved in \cite{BKK_RS:18} by considering  the global tensor 
decomposition of only the  "long-range part" in the tensor ${\bf P}_0$, defined by
\begin{equation}\label{eqn:Long-Range_Sum} 
 {\bf P}_l = \sum_{\nu=1}^{N} {z_\nu} \, {\cal W}_\nu (\widetilde{\mathbf{P}}_{R_l})=
 \sum_{\nu=1}^{N} {z_\nu} \, {\cal W}_\nu 
 (\sum\limits_{k= 1}^R \widetilde{\bf p}^{(1)}_k \otimes \widetilde{\bf p}^{(2)}_k 
 \otimes \widetilde{\bf p}^{(3)}_k).
 \end{equation}
 For tensor representation of the short-range part, ${\bf P}_s$, a sum of cumulative 
 tensors of small support (and small size) is used, accomplished by the list of weights and the 3D  
 coordinates of centers, $x_\nu$, of the single-particle potentials.
 
 In what follows, we recall the definition of the RS tensor format.
 \begin{definition}\label{Def:RS-Can_format} (RS-canonical tensors \cite{BKK_RS:18}). 
 Given the separation parameter $\sigma >0$,  the reference tensor ${\bf A}_0$, 
  where $\mbox{rank}({\bf A}_0)\leq R_0$ and $\mbox{diam}(\mbox{supp}{\bf A}_0)\leq 2 \sigma$,
 and a set of points $x_\nu \in \mathbb{R}^{d}$, and  weights $c_\nu$, $\nu=1,\ldots,N$.
 Then the RS-canonical tensor format specifies the class of $d$-tensors 
 ${\bf A}  \in \mathbb{R}^{n_1\times \cdots \times n_d}$
 which can be represented as a sum of a rank-${R}_L$ canonical tensor  
 \[
{\bf A}_{R_L} = {\sum}_{k =1}^{R_L} \xi_k {\bf a}_k^{(1)} \otimes \cdots \otimes {\bf a}_k^{(d)}
\in \mathbb{R}^{n_1\times ... \times n_d}
\]
and a cumulated canonical tensor 
 \[
 \widehat{\bf A}_S={\sum}_{\nu =1}^{N} c_\nu {\bf A}_\nu ,   
\]
generated by replication of the reference tensor ${\bf A}_0$ to the points $x_\nu$, 
${\bf A}_\nu= \mbox{Replica}_{x_\nu}({\bf A}_0)$.
Then the RS canonical tensor is presented by
\begin{equation}\label{eqn:RS_Can}
 {\bf A} =  {\bf A}_{R_L} + \widehat{\bf A}_S=
 {\sum}_{k =1}^{R_L} \xi_k {\bf a}_k^{(1)}  \otimes \cdots \otimes {\bf a}_k^{(d)} +
 {\sum}_{\nu =1}^{N} c_\nu {\bf A}_\nu. 
\end{equation}
\end{definition}

\begin{figure}[tbh]
\centering
\includegraphics[width=5.5cm]{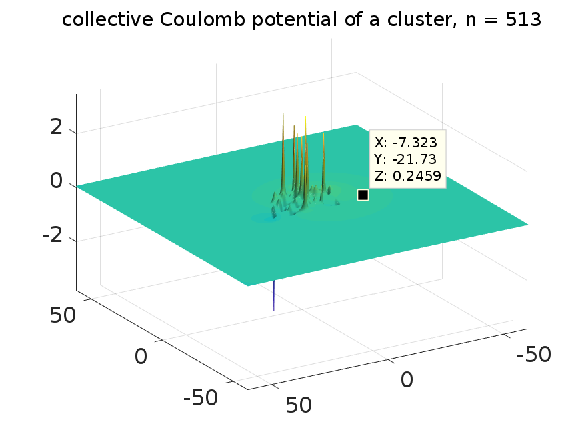}
\includegraphics[width=5.3cm]{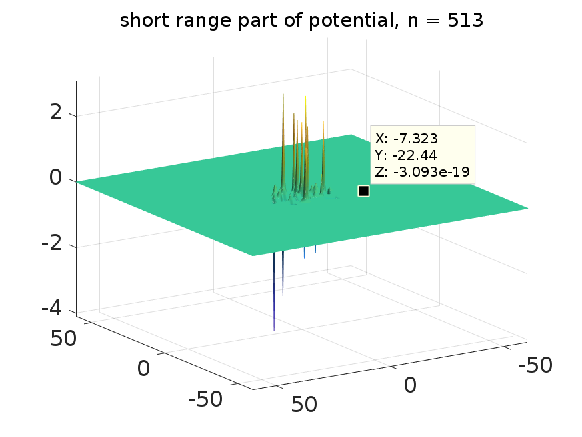}
\includegraphics[width=5.2cm]{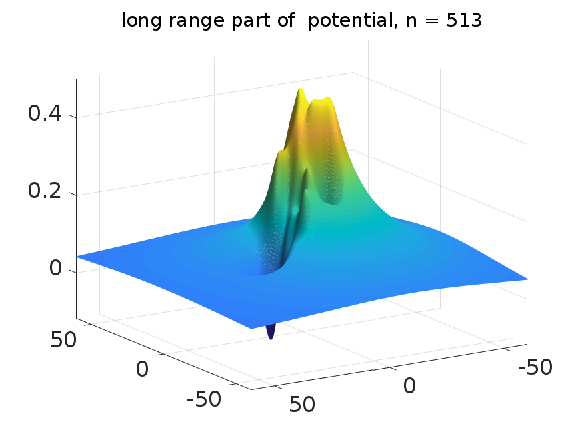}
\caption{\small
The collective electrostatic potential of a cluster with 782
charged particles (left), its short-range (middle) and the long-range (right) parts. }
\label{fig:Prot_long_1228}
\end{figure}
 Figure \ref{fig:Prot_long_1228} demonstrates the free space
 collective electrostatic potential of a cluster with 782
charged particles, and its short and long-range parts (see also \cite{BKK_RS:18})
 computed on the $n\times n\times n$ 3D Cartesian grid with $n=513$.

The storage size for  the  RS-canonical tensor ${\bf A}$ in (\ref{eqn:RS_Can}) 
is estimated by (\cite{BKK_RS:18}, Lemma 3.9),
$$
\mbox{stor}({\bf A})\leq d R_L n + (d+1)N + d n_\sigma R_0, \quad h n_\sigma \approx \sigma.
$$
Here $n_\sigma$ denotes the number of grid points in the interval $[0,\sigma]$ and $h>0$ is the mesh size.
Notice that in applications to numerical evaluation of electrostatic potential for bio-molecules
the constant $\sigma$ can be viewed as the van der Waals distance.

Denote by $\overline{\bf a}^{(\ell)}_{i_\ell}\in \mathbb{R}^{1\times R_L}$ the  
row-vector  with index $i_\ell$ in the side matrix $A^{(\ell)}\in \mathbb{R}^{n_\ell \times R_L}$
of ${\bf A}_{R_L}$, and by $\xi=(\xi_1,\ldots,\xi_d)$ the coefficient vector in (\ref{eqn:RS_Can}).
Then the ${\bf i}$-th entry of the RS-canonical tensor ${\bf A}=[a_{\bf i}]$ can be calculated 
as a sum of long- and short-range contributions, 
\[
 a_{\bf i}= \left(\odot_{\ell=1}^d \overline{\bf a}^{(\ell)}_{i_\ell} \right) \xi^T +
 \sum_{\nu: \, {\bf i} \in \mbox{\footnotesize supp}{\bf A}_\nu } c_\nu {\bf A}_\nu({\bf i}),
 \quad \mbox{at the expense}\quad O(d R_L + 2 d n_\sigma R_0).
\]
 
 In application to the calculation of multi-particle interaction potentials  discussed above
 we associate the tensors ${\bf P}_s$ and ${\bf P}_l$ in (\ref{eqn:Total_Sum}) 
 with short- and long-range components ${\bf A}_{R_L}$ and  $\widehat{\bf A}_S$ in the RS 
 representation of the collective electrostatic potential ${\bf P}_0$.
 The following theorem proofs the almost uniform in $N$ bound on the Tucker (and canonical) 
 rank of the tensor ${\bf A}_{R_L}$.
  \begin{theorem}\label{thm:Rank_LongRange}
 (Uniform rank bounds for the long-range part \cite{BKK_RS:18}).
Let the long-range part ${\bf P}_l$ in the total interaction potential, see (\ref{eqn:Long-Range_Sum}),
correspond to the sinc-approximation for generating radial function $p(\|x\|)$ with $K=O(\log^2\varepsilon)$, see (\ref{eqn:sinc}).
Then the total $\varepsilon$-rank ${\bf r}_0$ of the Tucker approximation to the canonical tensor sum ${\bf P}_l$
is bounded by
\begin{equation}\label{eqn:Rank_LongR}
 |{\bf r}_0|:=rank_{Tuck}({\bf P}_l) \leq C\, b \,\log^{3/2} (|\log (\varepsilon/N)|),
\end{equation}
where the constant $C$ does not depend on the number of particles $N$.
\end{theorem}

RS tensors provide the numerically efficient tool in many applications, for example for modeling the electrostatics of
many-particle systems of general type, or for modeling the scattered multidimensional data
by using radial basis functions, see examples in \cite{BKK_RS:18}.

\section{Fast energy calculation in posing process via the RS tensor decomposition of electrostatic potential}
 \label{sec:Problem_Docking}

 \subsection{Energy computation for clustered many-particle systems}   
 \label{ssec:Basics_Electrostatics}
 
The numerical simulation of the  space evolution for large many-particle systems is 
a commonly used tool 
in analysis of problems in the charged many-particle dynamics, docking (say, protein-ligand docking), 
identification and pattern matching of bio-molecules, as well in function prediction and drug discovery.
The numerical analysis can be based on the multiple evaluation of the binding energy 
for constrained minimization of the energy functional for the interacting particle system,
imposing the control of local geometric features. The latter requires the detailed learning 
of the complex geometric configurations in $\mathbb{R}^3$.

The time consuming energy evaluation (and the corresponding forces) 
should be performed many times in the course of corresponding 
optimization procedures in the multi-dimensional configuration space of possible rotations and translations 
parameterizing  the admissible transformations of a system. In the case of bio-molecular systems
in the presence of polarization effect, the complete analysis of the energy landscape  
requires the solution of the 3D Poisson-Boltzmann equation (PBE)
for each fixed point of interest in the multi-parametric configuration space. From computational 
point of view, this 
straightforward approach does not seem to be tractable even for moderate size system.

Recall that the collective electrostatic potential generated by 
the system of $N$ distinct charged particles with charges $z_k$ (which can be either positive or negative)
located at positions ${x}_{k}\in \mathbb{R}^3$ ($k=1,...,N$) is defined by 
\[
 P_0(x):= \sum\limits_{{k}=1}^N z_{k}p(\|{x} - {x}_{k}\|), \quad x\in \mathbb{R}^3,
\]
where the generating radial kernel function $p(r)=p(\|{x}\|)$ specifies the potential of the single particle with unit charge.
Here $r= \|{x}\|$ is the Euclidean norm of the vector $x\in \mathbb{R}^3$.
In this concern the free space electrostatic potential of the unit charge located at point $x_k\in \mathbb{R}^3$ 
is defined by the classical Newton kernel, $p(r)=1/r$,  i.e.
\begin{equation}\label{eqn:Newt_kern}
 p(\|{x} - {x}_{k}\|)= \frac{1 }{\|{x} - {x}_{k}\|}.
\end{equation}

The interaction energy of the system of charged particles ${\cal X}={\cal X}\{x_1,\ldots,x_N\}$ is 
defined by the weighted sum 
\begin{equation}\label{eqn:EnergyLatSum}
E_N =E_N(x_1,\dots,x_N)= \frac{1}{2} \sum\limits_{{j}=1}^N z_{j}
\sum\limits_{{k}=1, {k}\neq {j}}^N \frac{z_{k} }{\|{x}_j - {x}_{k}\|}. 
\end{equation}
Given the electrostatic potential $P_0(x)$,
the energy expression (\ref{eqn:EnergyLatSum}) can be rewritten in terms of the restricted (induced) 
non-singular potentials $P_{-j}(x)$, $ j=1,\ldots,N$, as follows
\begin{equation}\label{eqn:EnergybyPotential}
 E_N =\frac{1}{2} \sum\limits_{{j}=1}^N z_{j} P_{-j}(x_j),
\end{equation}
where $P_{-j}(x)$ given by 
\begin{equation}\label{eqn:Potentialminj}
 P_{-j}(x):= \sum\limits_{{k}=1, {k}\neq {j}}^N \frac{z_{k} }{\|{x} - {x}_{k}\|} 
 = P_0(x)- \frac{z_j}{\|{x} - {x}_{j}\|},
\end{equation}
defines the collective interaction potential generated by the reduced set of particles  
$
\{{\cal X} \setminus x_j\}, \quad j=1,\ldots,N.
$
In the case of long-range potentials, the 
straightforward calculation by (\ref{eqn:EnergyLatSum}) scales quadratically in the number of particles, 
$O(N^2)$.

The non-extensive numerical methods of $O(N)$ complexity for calculation of the interaction  energy (IE) for 
a charged multi-particle system given by (\ref{eqn:EnergybyPotential}) can be based on using the RS tensor format 
\cite{BKK_RS:18}, see \S\ref{sec:ten_formats}. 
Numerical schemes for evaluation of the potential energy (and force field) in the tensor formats will be discussed 
in \S\ref{ssec:Energy_RStensor}, \ref{ssec:Forces_Applic}. 
In the case of lattice-structured systems, 
the fast tensor-based computation scheme for IE is presented in \cite{Khor_bookQC_2018}.
In what follows,  
we describe the fast energy calculation  algorithm based on the RS tensor techniques,
that scales almost linearly in the number of particles, $O(N \log^q N)$.

\begin{remark} \label{rem:sigma-separable}
In this paper, we consider the $\sigma$-separable systems in the sense
that the minimal physical distance between the centers of particles is greater or equal 
than the fixed constant $\sigma >0$.
The double sum in (\ref{eqn:EnergyLatSum}) applies
only to the particle positions $\|{x}_{j} - {x}_{k}\|\geq \sigma $, hence, the quantity in 
(\ref{eqn:EnergyLatSum}) is correctly defined also for singular kernels 
like $p(r)=1/r^\alpha$, $\alpha >0$. 
\end{remark}
Recall that in bio-molecular electrostatic calculations the constant $\sigma >0$ is related to the so-called
van der Waals distance.

\begin{figure}[htb]
\centering
\includegraphics[width=4.0cm]{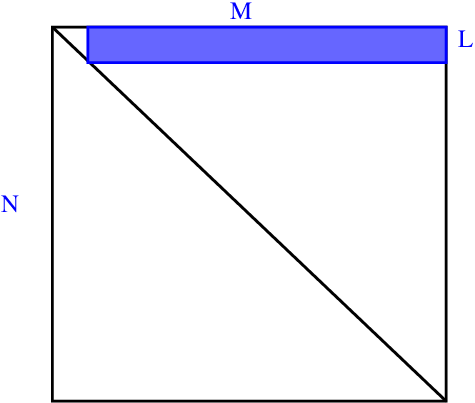}
\caption{\small Scheme of active summation indexes  in (\ref{eqn:EnergyLatSum_mixed}), (blue rectangle), 
where $N=M+L$.}
\label{fig:Summation_ind}
\end{figure}

Assume that the $\sigma$-separable system of charged particles  of interest, ${\cal X}$, is composed of two disjoint 
(non-overlapping) subsets, ${\cal X}= {\cal M} \cup {\cal L} $, and introduce notations
${\cal M}=\{x_m \}_{m=1}^M$, ${\cal L}=\{x_\ell \}_{\ell=1}^L$, with $N=M+L$.
 Then the total interaction energy defined by (\ref{eqn:EnergyLatSum}),
is represented by a sum of three different portions of  energy,
 \[
  E_N({\cal X})= E_M({\cal M}) + E_L({\cal L}) + E_N({\cal M},{\cal L}),
 \]
representing the respective energy of the subsystems ${\cal M}$ and ${\cal L}$,   as well as 
the interaction energy between these two disjoint subsystems (binding energy) defined by the weighted sum, 
\begin{equation}\label{eqn:EnergyLatSum_mixed}
E_N({\cal M},{\cal L})= \sum\limits_{{m}=1}^{M_{\cal M}} z_{m}
\sum\limits_{{\ell}=1}^{L_{\cal L}} \frac{z_{\ell} }{\|{x}_{m} - {x}_{\ell}\|} \equiv
\sum\limits_{{\ell}=1}^{L_{\cal L}} z_{\ell} \sum\limits_{{m}=1}^{M_{\cal M}} 
\frac{z_{m} }{\|{x}_{m} - {x}_{\ell}\|},
\end{equation}
where
\[
  M_{\cal M}=M:=\# {\cal M} \gg L_{\cal L}=L:= \# {\cal L}, \quad N=M+L.
 \]
Notice that in the course of 
dynamical process for optimization of the collective configuration in the rigid docking 
  only the interaction energy $E_N({\cal M},{\cal L})$ is changing and, hence, it should be updated after each 
geometric transformation of ligand ${\cal L}$. 

This allows to 
reduce the numerical costs from quadratic, $O(N^2)$, 
to linear, $O(M_{\cal M} L_{\cal L})$,
scaling in $M_{\cal M}$, taking into account that $M_{\cal M} \gg L_{\cal L}$, 
see Figure \ref{fig:Summation_ind} where 
the scheme of active set of summation indexes  of size $M\times L$ is presented. 
Indeed, computation of the full interaction energy for the cluster ${\cal X}= {\cal M} \cup {\cal L} $
requires evaluation of the potential over the indexes in the upper triangle 
of the $(M+L)\times (M +L)$ matrix, while 
in the case of two disjoint molecules the active summation index runs only over the small 
$M\times L$ rectangular subset.

Enhancing  the energy calculation by application of low-rank tensor decompositions requires
the representation of (\ref{eqn:EnergyLatSum_mixed}) in terms of the electrostatic potential of the system
similar to  (\ref{eqn:EnergybyPotential}). To that end we split the total electrostatic potential 
for two-cluster system into the sum of individual potentials for two subsystems as follows,
\begin{equation}\label{eqn:ElectrStat_pot}
 P_0(x):= \sum\limits_{{k}=1}^N \frac{z_{k}}{\|{x} - {x}_{k}\|}= P_{\cal M}(x) + P_{\cal L}(x) :=
 \sum\limits_{{m}=1}^M \frac{z_{m}}{\|{x} - {x}_{m}\|}+\sum\limits_{{\ell}=1}^L \frac{z_{\ell}}{\|{x} - {x}_{\ell}\|} , 
\end{equation}
with $x\in \mathbb{R}^3$, where the first sum on the right-hand side of (\ref{eqn:ElectrStat_pot}) 
runs over the particles in the bio-molecule $\cal M$, 
while the second one is taken over the atomic centers in ligand $\cal L$. 

Based on above energy splitting, the representation (\ref{eqn:EnergyLatSum_mixed}) 
for the cross-interaction (binding) energy 
for a system of two disjoint molecules takes one of the two equivalent forms
\begin{equation}\label{eqn:EnergySum_partial}
E_N({\cal M},{\cal L})=\sum\limits_{{\ell}=1}^{L_{\cal L}} z_\ell P_{\cal M}(x_\ell), 
\quad  (\, = \sum\limits_{{m}=1}^{M_{\cal M}} z_{m} P_{\cal L}(x_m)\,).
\end{equation}
It is worth to notice that numerical evaluation by both representations has the same asymptotic complexity,
$c_0 M_{\cal M}L_{\cal L}$, where the constant $c_0>0$ denotes the cost for calculation of the single Newton kernel
at each spacial point on the optimization trajectory in $\mathbb{R}^3$.

In the course of protein-ligand potential energy minimization the energy functional 
should be calculated many hundred if not 
thousand times in the case of large bio-molecules with thousands of atoms. 
This may lead to very high computational costs thus limiting the scope of target problems.

The question is how to reduce asymptotic cost of energy calculation in the perspective 
of large bio-molecule size $M \gg L$. 
In what follows, we propose the tensor based techniques that allow
to reduce the initial numerical cost $c_0 M_{\cal M}L_{\cal L}$ to the value $R\, L_{\cal L}$. 
Here $R=R_\Omega$ is the canonical rank of the 
CP tensor decomposition of the long-range part in the electrostatic potential ${\bf P}_{\cal M}$
discretized on the fine Cartesian grid in computational box $\Omega$. 
The CP rank is reduced to much smaller value $R_\Pi \ll R_\Omega$ for the long-range potential restricted onto the 
accompanying  minimal bounding box $\Pi:=[0,B]^{3}$ (moving $3$-hedron), see Table \ref{Tab:Times_ranks_3D},
that includes the small ligand  molecule ${\cal L}\subset \Pi$ at spacial positions of interest, 
see Figure \ref{fig:Bound_box}.  
\begin{figure}[htb]
\centering
\includegraphics[width=7.3cm]{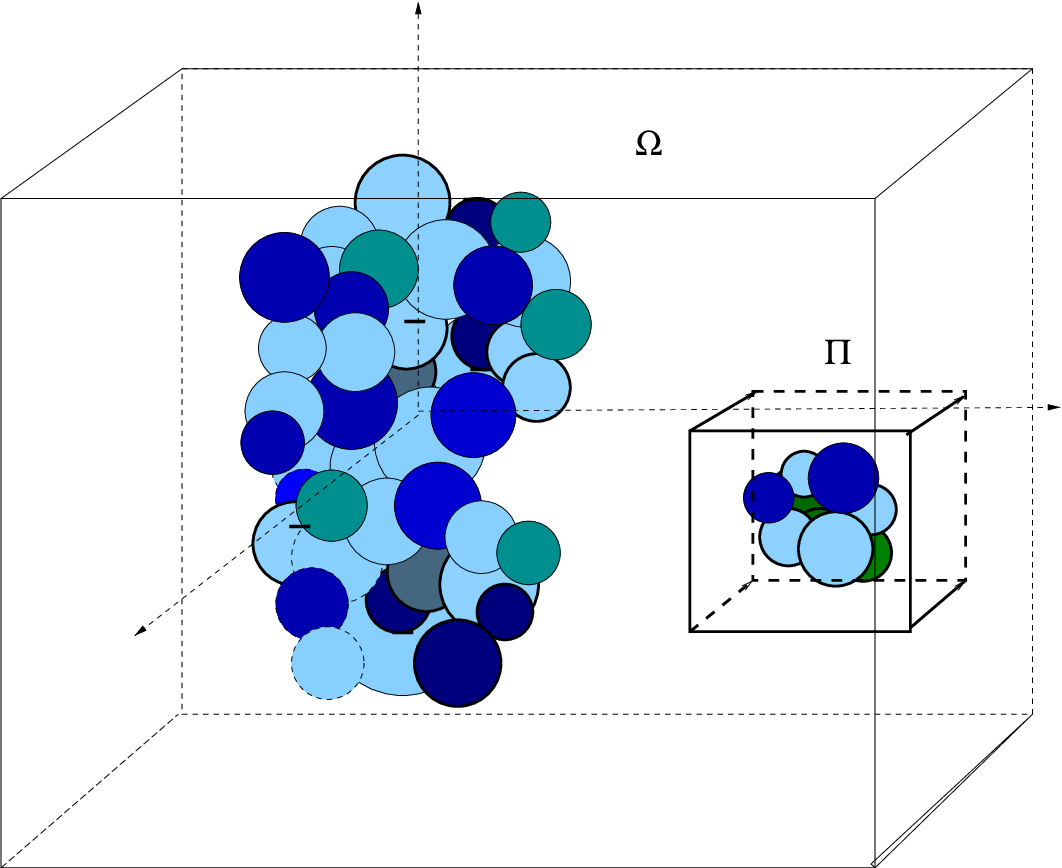}
\caption{\small Bounding box $\Pi$ for the small ligand and the containing hypercube $\Omega \supset \Pi$.}
\label{fig:Bound_box}
\end{figure} 

The mathematical background of this new approach is based on the rigorous analysis of the RS canonical tensor 
decomposition for many-particle electrostatic potentials, see \cite{BKK_RS:18}.
The main result in \cite{BKK_RS:18} proves a CP approximation for the long-range part in the potential, 
with the moderate rank parameter $R=R_\Omega$ that remains to be nearly independent of the number of 
  particles in a molecule, $M=M_{\cal M}$. Consequently, we arrive at the beneficial complexity estimate
  $Q(E_N)$
  for calculation of the electrostatic potential energy by (\ref{eqn:EnergySum_partial}), 
  $E_N({\cal M},{\cal L})$,
  $$
 Q(E_N) = R\, L_{\cal L} \ll c_0 M_{\cal M}L_{\cal L},\quad \mbox{where} \; R\ll M_{\cal M},
  $$ 
  at the limit of large $M_{\cal M}$.

  There are several possibilities to realize the above mentioned arguments, such that the final choice may depend 
  on the particular problem setting, required numerical precision, the computational budget and the amount of
  precomputed data. The respective computational schemes can be outlined as follows:
  \begin{itemize}
   \item[(A)] Precompute the $R_\Omega$-term tensor approximation of the long-range part 
   in the potential ${\bf P}_{\cal M}$ discretized in the bounding computational 
   hypercube $\Omega \supset {\cal M}\cup {\cal L}$, 
   and calculate its values at the respective locations of atoms in the target positions of moving ligand.
   \item[(B)] Compute the $R_\Pi$-term CP approximation to the potential $(P_{\cal M})_{| \Pi}$ that is
   the restriction of $P_{\cal M}$ onto the small box $\Pi \supset {\cal L}$, and then calculate the 
   electrostatic potential energy at the cost $Q(E_N)=R_\Pi \, L_{\cal L}$, where $R_\Pi \ll R_\Omega$.
   In this concern the external to ${\cal M}$ bounding box $\Pi$ might be large enough 
      such that the corresponding 
      low-rank representation could be used for many positions of ligand on the chosen trajectory. 
   \item[(C)] Hybrid algorithm: Once apply item (A) and then reduce the rank of the global $R_\Omega$-term CP 
   representation 
      to the small value $R_\Pi \ll R_\Omega$ in every bonding box $\Pi$ accompanying the 
      trajectory of moving ligand ${\cal L}$.
  \end{itemize}
  Each of numerical schemes (A) -- (C) has its own benefits and disadvantages. 
  For example, scheme (B) is prior to (A) in the case of large bio-molecules.
  Here we notice that the numerical cost for precomputing 
the low-rank tensor decomposition of the long-range part in the potential $P_{\cal M}$ scales 
linearly in the number of atoms in 
the protein, $M_{\cal M}$, and should be performed only once prior to implementing the posing/docking process.

  \begin{table}[htb]
\begin{center}\footnotesize
\begin{tabular}
[c]{|r|r|r|r|}%
\hline
cluster size & full CP rank  &  RS CP rank in $\Omega$&  RS CP rank in $\Pi$\\
\hline
16      &  224   &  51     &  4 \\
\hline
64     &   896   & 109      &  8 \\
\hline
100    &  1400   & 111     &  9 \\
\hline
256    &  3584  & 172     &  9 \\
\hline
400    &  5600  & 152 &  9 \\
\hline
676    &  9464  & 166  &  9 \\
\hline
\end{tabular}
\caption{\small RS long-range canonical ranks of the synthetic  flat ``molecular'' cluster constructed 
by random charges distribution and the ranks $R_\Pi$ of the potential restricted to the adjacent box $\Pi$ containing ligand. }
\label{Tab:Times_ranks_3D}
\end{center}
\end{table}
 
   Table \ref{Tab:Times_ranks_3D} demonstrates that when using the rotation procedure for the computation 
   of the ``ligand-protein'' 
   interaction energy for different rotation angles of the ``ligand''  one can use the  
   canonical tensor of small rank supported by the box $\Pi$ of small size
   located in the external to bio-molecule domain (solvent). 
   The rank parameter $R_\Pi$ remains practically independent of the size of the ``protein''.
  
  In what follows, we first consider the basic approach (A) that includes the RS tensor 
  approximation of the long-range part 
  in the molecular electrostatic potential ${P}_{\cal M}(x)$, discretized and precomputed in the large 
  bounding box $\Omega$. Afterward, the options (B) and (C) will be addressed.

 \subsection{Energy calculation by the RS tensor decomposition of  the interaction potential}   
 \label{ssec:Energy_RStensor}

We assume that the long-range part, ${P}_{\cal M}^{long}(x)$, of the 
potential ${P}_{\cal M}(x)$ is discretized on $n\times n \times n$ Cartesian grid in the hypercube 
$\Omega  \supset \Pi$, see Figure \ref{fig:Bound_box}, that includes 
the whole molecular-ligand system of interest, ${\cal M} \cup {\cal L}$. The respective rank-$R$ CP approximation of 
${P}_{\cal M}^{long}(x)$, further denoted by ${\bf P}_{{\cal M},R}^{long}$, 
is supposed to be already precomputed \cite{BKK_RS:18}. 

In what follows, 
we discuss the mathematical details of the energy and force calculation by using RS tensor approximation
for the case of two separated interacting molecular clusters.
We outline the particular details for numerical implementation of the
scheme described by the energy representation  (\ref{eqn:EnergySum_partial})
and applied in section \ref{ssec:Docking_Energy}.

First, we sketch the general theory of energy computation by using only the long-range part
of electrostatic potential as presented in \cite{BKK_RS:18}. Recall that the reference canonical tensor
${\bf P}_R$ approximating the single Newton kernel on an $n\times n\times n $ tensor grid $\Omega_h$
in the computational box $\Omega=[-b,b]^3$ is represented by (\ref{eqn:sinc_general}), where $h>0$ 
is the fine grid size.
For ease of exposition, we assume that the particle centers ${x}_{k}$
are located at some grid points in $\Omega_h$
(otherwise, an additional approximation error may be introduced) such that
each point ${x}_{k}\in \Omega_h$ inherits some multi-index ${\bf i}_{k}\in {\cal I}:=\{1,...,n\}^{\otimes 3}$, 
and the origin
$x=0$ corresponds to  the central point $x^{(0)}$ on the grid indexed by ${\bf n}_0=(n/2,n/2,n/2)$.

In turn, given the short- and long-range parts in the reference tensor ${\bf P}_R={\bf P}_{R_s} + {\bf P}_{R_l}$,
then the canonical tensor ${\bf P}_0$ approximating the total interaction potential 
$P_0(x)  \equiv {P}_{\cal M}(x)$ ($x\in \Omega_h \subset \Omega$) for the $N$-particle system
is represented as a sum of short- and long-range tensor components \cite{BKK_RS:18}, see also \S\ref{sec:ten_formats},
\[
 P_0(x)=\sum\limits_{{k}=1}^N \frac{z_{k} }{\|{x} - {x}_{k}\|}\, \leadsto \,
 {\bf P}_0 = {\bf P}_s + {\bf P}_l\in \mathbb{R}^{n\times n\times n}.
\]
 By construction the tensor ${\bf P}_0={\bf P}_0(x^h)$ is defined at the vicinity of every grid-point $x^h \in \Omega_h$, 
 in particular at the grid-points $x_k + h{\bf e}$, where
the directional vector ${\bf e}=(e_1,e_2,e_3)^T$ is specified by the particular choice of 
the shifting parameters in the set $e_\ell \in\{-1,0,1\}$ for $\ell=1,2,3$. 
This allows us to introduce the useful notation 
${\bf P}_0(x_k + h{\bf e})$ which can be applied to all tensors leaving on $\Omega_h$.
Such notation simplifies the definition of such entities like energy, gradients, forces, etc. (in terms of given 
interaction potential)
applied to the RS tensors in the course of dynamical energy minimization.

The following statement describes the tensor scheme for calculation the interaction energy $E_N$ by 
utilizing the only long-range part ${\bf P}_l$ in the tensor representation ${\bf P}_0$ of $P_0(x)$. 
  
\begin{proposition}\label{lem:InterEnergy} (\cite{BKK_RS:18})
 Let the effective support of the short-range components in the reference
 potential ${\bf P}_R$ does not exceed $\sigma>0$.
 Then the interaction energy $E_N$ of the $N$-particle $\sigma$-separable system given by (\ref{eqn:EnergyLatSum}) can be
 calculated by using only the long range part, ${\bf P}_l$, in the total potential sum as follows,
 \begin{equation}\label{eqn:EnergyFree_Tensor}
 E_N =E_N(x_1,\dots,x_N)= \frac{1}{2} 
 \sum\limits_{{j}=1}^N z_{j}({\bf P}_l({x}_{j}) - z_j {\bf P}_{R_l}(x^{(0)})),
\end{equation}
in $O(d R N)$ operations, where $R$ is the canonical rank of the long-range component ${\bf P}_l$.
\end{proposition}
\begin{proof} We sketch the proof for the sake of completeness. 
Similar to \cite{Khor_bookQC_2018}, where the case of lattice structured systems was analyzed,
we show that the interior sum in (\ref{eqn:EnergyLatSum})
can be obtained by using  the tensor ${\bf P}_0$
traced onto the centers of particles ${x}_{k}$, such that
the ''self-interaction`` term corresponding to ${x}_{j}={x}_{k}$ is removed, 
\[
 P_{N \backslash  j}(x_j) = \sum\limits_{{k}=1, {k}\neq {j}}^N \frac{z_{k} }{\|{x}_{j} - {x}_{k}\|}
  \, \leadsto \, {\bf P}_0({x}_{j}) - z_{j}{\bf P}_R(x^{(0)}).
\]
Here the reference canonical tensor ${\bf P}_R$ is evaluated at the origin $x^{(0)}=0$
and the potential $P_{N \backslash  j}(x)$ is defined by (\ref{eqn:Potentialminj}). 
Hence, we arrive at the tensor approximation of the interaction energy
\begin{equation}\label{eqn:EnergyFree_TensF}
 E_N = 
 \frac{1}{2} \sum\limits_{{j}=1}^N z_j P_{N \backslash  j}(x_j) \, \leadsto \,
 \frac{1}{2} \sum\limits_{{j}=1}^N z_{j}({\bf P}_0({x}_{j}) - z_j {\bf P}_R(x^{(0)}).
\end{equation}
Now we split ${\bf P}_0$ into the long-range part (\ref{eqn:Long-Range_Sum})
and the remaining short-range potential, to obtain
${\bf P}_0({x}_{j})= {\bf P}_s({x}_{j}) + {\bf P}_l({x}_{j})$, and the same for the reference
tensor ${\bf P}_R$.
By assumption, the short-range part ${\bf P}_s({x}_{j})$ at point ${x}_{j}$ 
in (\ref{eqn:EnergyFree_TensF}) includes only  the contribution from the 
local term corresponding to the atomic location at $x_j$, i.e.  ${\bf P}_{s}(x_j)=z_j {\bf P}_{R_s}(x^{(0)})$.
Due to the corresponding cancellations on the right-hand side of
(\ref{eqn:EnergyFree_TensF}), we find that $E_N$ depends only on ${\bf P}_l$,
which proves the final tensor representation in (\ref{eqn:EnergyFree_Tensor}).

The linear complexity scaling of the order of $O(d R N)$ is proven by
taking into account the $O(d R )$-cost of the point evaluation for the long-range rank-$R$ canonical 
tensor ${\bf P}_l$.
\end{proof}

It was already noticed in \S\ref{ssec:Basics_Electrostatics}, 
see (\ref{eqn:EnergyLatSum_mixed}), (\ref{eqn:EnergySum_partial}),
that the complexity of the above calculation can be reduced 
by splitting of the total energy into three terms such that the only interaction energy (binding energy) between 
two clusters ${\cal M}$ and ${\cal L}$, $E_N({\cal M},{\cal L})$, 
should be evaluated in the course of energy functional minimization, taking into account that $M \gg L$.
\begin{lemma}\label{lem:energy_ligand}
(Fast binding energy calculation).
 Given the rank-$R$ long-range component in ${\bf P}_{\cal M}$,  
 ${\bf P}_{{\cal M},R}^{long}$, then the binding energy of each molecular-ligand configuration given by (\ref{eqn:EnergySum_partial})
 can be calculated in $(d-1) R\, L_{\cal L}$ operations as follows,
  \begin{equation}\label{eqn:EnergySum_ligand}
E_N({\cal M},{\cal L})= 
 \sum\limits_{{\ell}=1}^{L_{\cal L}} z_\ell {\bf P}_{{\cal M},R}^{long}(x_\ell),\quad x_\ell \in {\cal L},
\end{equation}
 where $L_{\cal L}$ is the number of atoms in ligand.
\end{lemma}
\begin{proof}
 We chose the binding energy representation in the form (\ref{eqn:EnergySum_partial}), 
 \begin{equation}\label{eqn:EnergySum_partial_lig}
E_N({\cal M},{\cal L})= \sum\limits_{{\ell}=1}^{L_{\cal L}} z_\ell P_{\cal M}(x_\ell),\quad x_\ell \in {\cal L},
\end{equation} 
 and substitute the exact 
 potential $P_{\cal M}(x)$ by its discretized long-range part approximated by rank-$R$ canonical 
 tensor ${\bf P}_{{\cal M},R}^{long}$ to obtain (\ref{eqn:EnergySum_ligand}).
 Indeed, the latter is just the discretized version of (\ref{eqn:EnergySum_partial}), written in the form
 (\ref{eqn:EnergySum_partial_lig}),
taking into account the argument of Proposition \ref{lem:InterEnergy} stating that the short-range contribution to 
$P_{\cal M}(x)$, $x \in {\cal L}$ does not effect the energy expression in (\ref{eqn:EnergySum_partial_lig}).

Finally, we notice that the pointwise calculation of the rank-$R$ canonical tensor in $d$ dimensions 
amounts to $(d-1)R$ arithmetic operations, then the complexity result $(d-1) R\, L_{\cal L}$ follows from the 
representation (\ref{eqn:EnergySum_ligand}).
\end{proof}

Lemma \ref{lem:energy_ligand} leads to the important conclusion that application of the RS tensor decompositions 
of the electrostatic potential of bio-molecule
allows an implementation of one step in the posing minimization process (i.e., the energy calculation for current 
configuration) at the numerical cost that
does not depend on the size of the target bio-molecule. 
It depends on the quality of the precomputed data, i.e. optimality of the rank parameter $R$. 

Table \ref{Tab:Times_check_3D} demonstrates the accuracy of energy calculation via (\ref{eqn:EnergySum_ligand})
depending on the initial rank parameter $R$ of the complete tensor and on the grid size $n$.
The results are valid for the full rank-$R$ tensor and the long-range tensor of rank equal to $R_L=16$.
\begin{table}[htb]
\begin{center}\footnotesize
\begin{tabular}
[c]{|r|r|r|r|}%
\hline
grid size $n^3$  & initial CP rank  & ${E_N}_{Tensor}$  & $E_N - {E_N}_{Tensor}$   \\
\hline  
$128^3$ &  20  &  -0.2511    & 0.0154       \\ 
\hline
$256^3$ &  21  &  -0.2432    & 0.0075         \\ 
\hline
$512^3$ & 24   &  -0.2394    & 0.0037        \\ 
\hline 
$1024^3$ &  26 &  -0.2376   & 0.0019       \\ 
\hline
$2048^3$ & 27  &  -0.2366  & $9.3\cdot 10^{-4}$ \\ 
\hline
$4096^3$ & 29  &  -0.2362  & $4.6\cdot 10^{-4}$     \\ 
\hline
$8192^3$ & 31  &  -0.2358  & $2.3\cdot 10^{-4}$  \\ 
\hline
$16384^3$ & 32  &  -0.2358  & $1.15\cdot 10^{-4}$  \\ 
\hline
\end{tabular}
\caption{\small Comparison of direct and tensor computations of the electrostatic interaction  
energy of two point charges at the distance
of 3 atomic units (approximately 1.5 \AA{}). Direct computation results in $E_N=-0.2357$ hartree. }
\label{Tab:Times_check_3D}
\end{center}
\end{table}

Recall that the complexity of the direct energy calculation by (\ref{eqn:EnergyLatSum_mixed}) 
amounts to $c_0 M L$ ($c_0$ is the cost for calculation of the single Newton kernel) 
for each position of ligand,
and it likely becomes highly expensive in the course of the energy minimization 
for large size $M$ of the bio-molecule.
Rigorous description, mathematical analysis and numerical verification of the beneficial tensor-based computational
techniques is one of the main results of this paper.

\subsection{Gradient and force field calculations}\label{ssec:Forces_Applic}

 Computation of electrostatic forces in and gradients of interaction potential 
 in multi-particle systems is a computationally extensive problem.
The algorithms based on Ewald summation technique \cite{Ewald:27} are considered as the basic 
approaches in the literature, see for example the discussion  in \cite{HoEa:88}.
Here we recall the recently introduced alternative approach to force calculation on a particle 
by using RS tensor format for representation of the electrostatic potential involved 
in the energy representation (\ref{eqn:EnergySum_ligand}). 

In view of representation (\ref{eqn:EnergySum_ligand}) for the energy, first, we consider the computation 
of gradient with respect to the variation in the position of ligand.
Given an RS-canonical tensor ${\bf A}={\bf P}_{\cal M}$ as in (\ref{eqn:RS_Can}) with 
the width parameter $\sigma >0$, 
the discrete gradient $\nabla_h=(\nabla_1,\ldots,\nabla_d)^T$ applied to the long-range
part in ${\bf P}_{\cal M}$, ${\bf A}_l = {\bf P}_{{\cal M},R}^{long}$, at the grid points of $\Omega_h$  
can be calculated in the form of the $R$-term canonical tensor by applying the simple
one-dimensional finite-difference (FD) operations to the long-range part 
in ${\bf P}_{\cal M}={\bf A}_s + {\bf A}_l$,
\begin{equation}\label{eqn:Gradients_Tens}
{\bf A}_l={\sum}_{k =1}^{R} \xi_k \, {\bf a}_k^{(1)}  \otimes \cdots \otimes {\bf a}_k^{(d)},
\quad \nabla_h {\bf A}_l= {\sum}_{k =1}^{R} \xi_k ({\bf G}_k^{(1)},\ldots,{\bf G}_k^{(d)})^T,
\end{equation}
with tensor entries at any fixed grid point in the accompanying box $\Pi_h \supset {\cal L}$, 
\[
 {\bf G}_k^{(p)} ={\bf a}_k^{(1)}  \otimes \cdots\otimes 
 \nabla_{p,h} {\bf a}_k^{(p)} \otimes \cdots\otimes{\bf a}_k^{(d)},
\]
where $\nabla_{p,h}$ ($p=1,\dots,d$) is the univariate FD differentiation scheme
(by using backward or central differences). Numerical complexity of
the representation (\ref{eqn:Gradients_Tens}) at each grid point can be estimated by $O(d R )$, where 
the canonical rank $R$ is almost uniformly bounded in the number of particles $M$ in the bio-molecule ${\cal M}$.

Furthermore, 
the force vector $ {\bf F}_\ell \in \mathbb{R}^d$ (as well as the gradient) on the particle located 
at $x_\ell \in {\cal L}$ is obtained by differentiating the 
electrostatic potential binding energy $E_N({\cal M},{\cal L}) = E_N(x_1,\dots,x_L)$, 
given by (\ref{eqn:EnergySum_ligand}), with respect to $x_1,\dots,x_L$ at points $x_\ell$, 
\[
 {\bf F}_\ell=-\frac{\partial}{\partial x_\ell} E_N = - \nabla_{| x_\ell} E_N, \quad \ell=1,\ldots,L.
\]
In the general case of full energy expression, 
$
  E_N({\cal X})= E_M({\cal M}) + E_L({\cal L}) + E_N({\cal M},{\cal L}),
$
it can be calculated explicitly (see \cite{HoEa:88}) in the form,
\[
 {\bf F}_\ell = \frac{1}{2} z_{\ell} \sum\limits_{{k}=1, {k}\neq {\ell}}^N  z_{k} 
 \frac{{x}_{\ell} - {x}_{k} }{\|{x}_{\ell} - {x}_{k}\|^3}\in \mathbb{R}^d,\quad \ell=1,\ldots,L.
\]

In our special case of protein-ligand rigid docking we have to apply the gradient operation 
only to the non-stationary energy term, 
$E_N({\cal M},{\cal L})$, where the latter is given by (\ref{eqn:EnergySum_ligand}). This leads to
the expression
\[
 {\bf F}_\ell = - \nabla_{| x_\ell} E_N({\cal M},{\cal L}) =
 z_{\ell} \sum\limits_{{m}=1}^M  z_{m} 
 \frac{{x}_{\ell} - {x}_{m} }{\|{x}_{\ell} - {x}_{m}\|^3} \equiv z_{\ell} {\bf E}_M(x_\ell),\quad
 {x}_{\ell}\in {\cal L},\quad \ell=1,\ldots,L,
\]
where the corresponding vector potential (electric vector field) is defined by
\[
{\bf E}_M(x) :=  \nabla_{| x} {\cal P}_{\cal M}(x)  =  
\sum\limits_{{m}=1}^M  z_{m} \nabla_{| x} \frac{1}{\|{x} - {x}_{m}\|}=
\sum\limits_{{m}=1}^M  z_{m} 
 \frac{{x} - {x}_{m} }{\|{x} - {x}_{m}\|^3}, \quad {x}_{m} \in {\cal M}, \quad x \in {\cal L}.
\]
It is possible to construct the RS tensor representation for the vector 
field ${\bf E}_M(x)$ directly by combining the RS representation of radial basis functions 
$p(r)=1/r$ and $p(r)=1/r^2$. Another option is the application of gradient $\nabla_{| x}$ 
in analytic form to the initial representation of the long-range part in the Newton kernel
$\frac{1}{\|{x}\|}$ by a sum of Gaussians. 
This issue will be addressed elsewhere.

In the rest of this section, we describe the finite difference approach based on 
numerical differentiation of the energy functional based on RS tensor representation 
of the $M$-particle interaction potential ${\bf P}_{{\cal M},R}$ on fine spacial grid.
The differentiation of the long-range potential ${\bf P}_{{\cal M},R}^{long}$ with respect to arbitrary position, 
say for $x_\ell=x_L$, 
is based on the explicit representation (\ref{eqn:EnergySum_ligand}) in Lemma \ref{lem:energy_ligand},
which can be used for the discrete differentiation in $\mathbb{R}^d$, $d=3$. 
 Specifically, when approximating the force field vector at point $x=x_L$, 
${\bf F}_L=\{F_{L,1},F_{L,2},F_{L,3}\}\in \mathbb{R}^3$, 
we arrive at the powerful representation for 
the first differences in direction ${\bf e}_i$ for $i=1,2,3$,
\begin{equation}\label{eqn:differ_ligand}
  F_{L,i} \approx  \frac{1}{h}\left( E_N(\cdot,x_L) - 
  E_N(\cdot,x_L-h{\bf e}_i)\right) = 
 \frac{z_L}{h}\left( {\bf P}_{{\cal M},R}^{long}(x_L) -
 {\bf P}_{{\cal M},R}^{long}(x_L -h{\bf e}_i)\right),
\end{equation}
with the abbreviation $E_N(\cdot,x_L) := E_N(x_1,\ldots,x_L)$.

The straightforward implementation of the above relation for three different values of
${\bf e}_1=(1,0,0)^T$, ${\bf e}_2=(0,1,0)^T$ and ${\bf e}_3=(0,0,1)^T$
is reduced to the $(d+1)$ calls of the basic
procedure for computation the tensor ${\bf P}_{{\cal M},R}^{long}$ corresponding to $d=3$ spatial directions thus 
leading to the total cost $O((d-1)(d+1) R)$ almost independently on the size $M$ of bio-molecule ${\cal M}$.

The above techniques apply to force calculation on a single charged particle, \cite{BKK_RS:18}.
In Section 4, we describe the new approach for force calculation 
applicable to the case of rigid cluster of charged particles, for instance in case of ligand molecule.

\subsection{Posing as constrained energy minimization}
\label{ssec:Docking_Energy}


 Calculation of the appropriate protein-ligand posing positions plays the central role
in predicting of binding conformations in the rigid docking process. 
In our numerical scheme, we treat the posing process as constrained minimization of the effective
electrostatic potential energy functional of the protein-ligand ensemble, 
that is finding the local minimum points on the potential energy surface (binding energy).

We numerically study the energy landscape  of the complex particle
system ${\cal X}$ of interest, composed of two
non-overlapping subsets, ${\cal X}= {\cal M} \cup {\cal L} $ 
(large bio-molecule $ {\cal M}$ and small ligand ${\cal L}$),
by means of minimization the energy functional, see (\ref{eqn:EnergySum_partial}),
$$ E_N({\cal X}) =E_N({\cal M},{\cal L})=  E_N({\cal M},x_1,\dots,x_L),
$$ 
under certain
constraints on the mutual location of rigid atomic configurations ${\cal M}, {\cal L} \subset {\cal X}$, that is
they need to be separated by the so-called van-der-Waals distance $\sigma>0$. 

We suppose that position of the bio-molecule ${\cal M}$ is fixed, while ligand ${\cal L}$ 
is allowed for 3D geometric movements ${\cal F}={\cal R}\circ{\cal T}$ composed of
a superposition of translation ${\cal T}$ applied to the geometric center $x_0$ of the ligand,
and rotation ${\cal R}$ in $\mathbb{R}^3$ around this center. 
Then the problem is reduced to finding a posing positions
of ligand that realizes the nearly minimal value of the interaction energy $E_N({\cal X})$ 
under the assumptions in Remark \ref{rem:sigma-separable}.

Now each geometric transformation ${\cal F}$ is uniquely defined by parameter vector ${\bf q}=(x_0,\phi)$,
where $x_0=(x,y,z)\in \mathbb{R}^3$ can be associated with the coordinates of geometric 
''center of mass`` for ${\cal L}$ and $\phi=(\alpha, \beta)$
denotes the rotational angles in $\mathbb{R}^3$ around $x_0=(x,y,z)$. We may assume that the positions 
of atoms in both ${\cal M}$, and ${\cal L}$
attain the representation on the Cartesian $n \times n\times n$ grid $\Omega_h$ such that $x_0\in\Omega_h $
lives on that grid.
Then the minimization problem can be formulated as finding the best fit in terms
of the interaction energy,
\begin{equation}\label{eqn:Energy_Min}
 [ x_0^\ast,\phi^\ast ] = \mbox{arg min}_{x_0,\phi} E_N[{\cal M},{\cal F}_{x_0,\phi}({\cal L})]
 \quad s.t. \quad  \mbox{dist} ({\cal M},{\cal F}({\cal L}))\geq \sigma>0,
\end{equation}
where the latter constraint means the non-overlapping requirements for any of predicted poses. 
This leads to the simplified version of the so-called scoring function in our case, see e.g. \cite{Nature_Rev:2004}.
In this concept, we suppose that the geometry of both configurations ${\cal M}$ and ${\cal L} $
remains fixed and only rotations and translations of ${\cal L} $  in $\mathbb{R}^3$ 
are allowed (rigid docking).

In the particular case of minimization problem (\ref{eqn:Energy_Min}) the energy functional 
can be discretized on the grid $\Omega_P$ in five 
dimensional parametric space for vectors ${\bf q}=(x_0,\phi)$, that includes the three spacial 
parameters specifying the translation variables, i.e. coordinates of $x_0=(x,y,z)$, 
and two parameters ${\phi}$ defining the rotation angles of ligand
in $\mathbb{R}^3$ around fixed point $x_0$. This leads to the challenging 
minimization problem in five-dimensional parametric space, $\mathbb{R}^5$, where each functional 
evaluation might be expensive. 
For example, the fine discretization with $m=50$ sampling points in each of five parameters leads to the huge 
size of about $m^5=50^{5}$ samples for the discretized parametric space of size $O(m^{{\otimes 5}})$, where 
PES has to be investigated. Hence, the full search in $\mathbb{R}^{5}$ quickly becomes 
non-tractable.

To reduce the numerical cost, we propose to apply (a) the low rank tensor representation of the effective 
interaction potential $P_{\cal M}$
and subsequently the super-fast calculation of the interaction energy, (b)
the use, if necessary, a sequence of dyadic refined representation grids $\Omega_h$ to lift-up 
the chipper information from the coarse to fine levels, (c) introducing a sequence of reference grids $\Omega_P$
in parametric space, 
and (d) certain selection of the most appropriate rotation positions (kind of supervised learning) as well as 
(e) reduce the geometric area for location of ligand ${\cal L}$ in the course of energy minimization.

Figure \ref{fig:layer_ind} illustrates the example of layer-like enveloping ring where the energy optimization 
has to be implemented, see item (d) above.
\begin{figure}[htb]
\centering
\includegraphics[width=6.0cm]{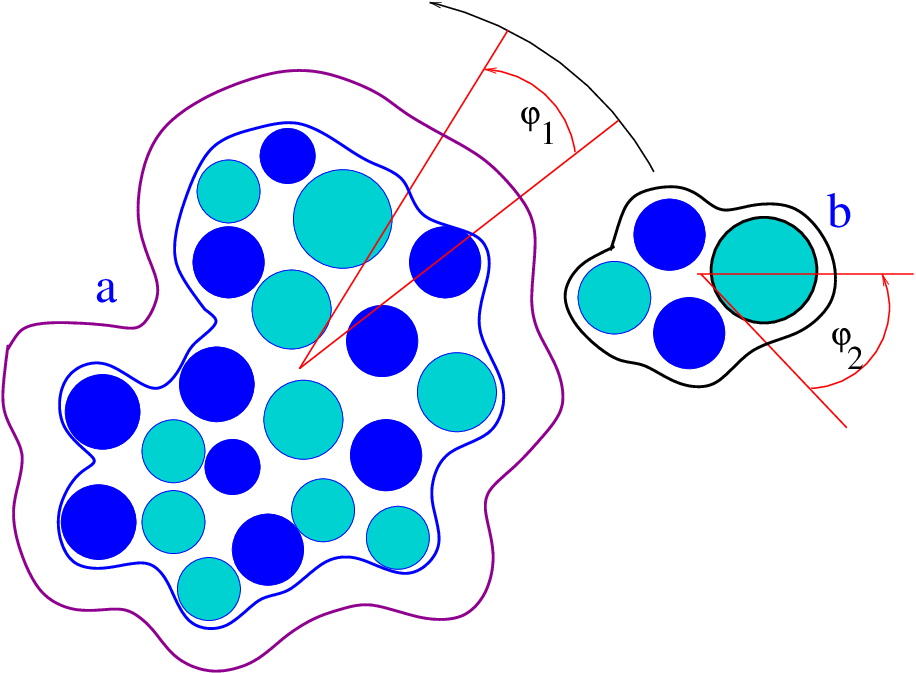}
\caption{\small Schematic illustration of the layer-like enveloping ring where the energy optimization 
has to be implemented.}
\label{fig:layer_ind}
\end{figure} 

Notice that in the case of protein-ligand rigid docking problem the fixed object (protein) ${\cal M}$
 may have extremely complicated profile, i.e. the landscape of the admissible part 
 of the potential energy surface (PES), while the moving part (ligand) ${\cal L}$ has relatively simple 
 geometry and may include the moderate 
 number of particles. Hence, in the following complexity discussions we assume that 
 \[
  N_{\cal M}:=\# {\cal M} \gg N_{\cal L} := \# {\cal L}.
 \] 
 The complex shape of the big target molecule may include several small pieces
which could be considered as good candidates for the search of local 
minimum of the scoring functional in (\ref{eqn:Energy_Min}).  
This means that the optimized and 
robust algorithms to approximate solution(s) of this global minimization problem 
should be only applied to several smaller sub-problems, which solve the problem
(\ref{eqn:Energy_Min}) patch-wise by investigating the local minimum 
around  the selected small pieces of the molecular surface. In this approach the 
initial binding poses could be generated manually by optimizing the scoring function within some cylinder (cone) 
with the axis directed to the ``geometric center'' of the bio-molecule
and  intersecting the ``host'' piece of the molecule (host-guest binding).

Recall that in the course of 
transformation process of ligand for optimization of the collective configuration in the rigid docking 
the only interaction part of the full energy, $E_N({\cal M},{\cal L})$ (binding energy), 
is changing, see (\ref{eqn:EnergyLatSum_mixed}),  
hence, the only $E_N({\cal M},{\cal L})$ should be updated after each 
geometric transformation of ligand, see Lemma \ref{lem:energy_ligand}.

 Our numerical scheme for modeling of the protein-ligand docking includes three main blocks: \\
(A1) Algorithm for calculation of the low-rank tensor-based representation electrostatic potential (TEP) of protein discretized on 3D Cartesian grid in the exterior of ${\cal M}$. 
Application of Algorithm TEP (see below) is considered as the pre-computing step
to be performed only once before the beginning of energy minimization process.\\
(A2) Tensor-based Algorithm for fast calculation of 
the protein-ligand (electrostatic) interaction energy (PLIE) 
by using the once pre-computed rank-structured electrostatic potential of the bio-molecule ${\cal M}$, represented on 
$n\times n\times n$ Cartesian grid by using Algorithm TEP below.\\
(A3) Algorithm for finding the appropriate posing positions via
interaction energy minimization (PPIEM), see (\ref{eqn:Energy_Min}), 
which includes multiple calls to the Algorithm PLIE in the course of minimization
trajectory of the ligand.

First, we describe the Algorithm TEP to be used for calculation  of the  electrostatic energy in the CP (canonical) tensor parametrization. 

{\bf Algorithm TEP.} The low-rank representation for electrostatic potential of ${\cal M}$ in $\Omega_h$.
  \begin{enumerate}
  \item \emph{Initialization: Given positions and charges of all atoms in ${\cal M}$ and in ${\cal L}$, grid size $n$
  and the van der Waals constant $\sigma >0$. 
  Construct the $n\times n\times n$ Cartesian representation grid that envelops the domain ${\cal M}\bigcup {\cal L}$}. 
  
  \item \emph{Generate the rank-$R_L$ canonical reference tensor representation ${\bf N}_L$ 
  for the long-range part in the Newton kernel as a sum of long-range 
  Gaussians living on  $2n\times 2n\times 2n$ 3D Cartesian grid, with complexity $O(n)$,  \cite{BeHaKh:08}.
 Use the van der Waals distance $\sigma >0$ to control from below the support of long-range Gaussians.} 
 
  \item \emph{Perform summation of shifted copies of the long-range Gaussians in the reference 
  tensor traced onto $n\times n\times n$ 3D Cartesian representation grid, 
  by their replication to the centers of charged particles. 
  The cost of the respective one-dimensional shifts of canonical vectors is $O(R_L)$ for every particle.
  Thus, we obtain the long-range part of collective free space electrostatic potential 
  for a protein with the initial rank $R' = N_{\cal M} R_L$, 
  where $N_{\cal M}$ is the number of particles in the protein.} 
  
  \item  \emph{Reduce the rank of the resulting canonical tensor to $R\ll R'$ 
  by the canonical-Tucker-canonical transform with possible application of add-and-compress techniques.  } 
  
  \item  \emph{Output: the long-range part of the free-space collective
  electrostatic potential of the protein, ${\bf P}_{{\cal M},R}^{long}$, represented on the whole 
  $n\times n\times n$ 3D grid at the cost $O(N_{\cal M} n)$.}
 \end{enumerate}

The CP tensor representation of the free-space electrostatic potential of the protein  obtained by procedure TEP
is the prerequisite for fast computation of  electrostatic interaction energy 
(PLIE) for protein-ligand complex, given a position of the ``protein'' and ``ligand''. 
Sketch of the Algorithm PLIE is presented as follows.

{\bf Algorithm PLIE.} Fast calculation of the protein-ligand interaction energy.
  \begin{enumerate}
   \item \emph{Given the collective long-range electrostatic potential of the bio-molecule,
   ${\bf P}_{{\cal M},R}^{long}$, represented in the rank-$R$ CP format on the whole computational grid $\Omega_h$ 
   (pre-computed by algorithm TEP), as well as charges $q_i$ and  coordinates of the particles 
   in the ligand, $(x_i,y_i,z_i)_{i=1:N_L}$.}
   
   \item \emph{Sample the $q_i$-wighted potential ${\bf P}_{{\cal M},R}^{long}$ at the centers 
   of the ligand 
   $(x_i,y_i,z_i)_{i=1:N_L}$, by using the canonical-to-point algorithm to obtain a vector 
   ${\bf p}_{\cal L}$  of size $N_L$, composed of these point-wise  values of the potential. 
   The cost of the operation is $O(R\, N_L)$. }
   
   \item \emph{Construct the vector of charges of the particles in the ligand, 
   ${\bf z}_{\cal L}$, and compute the target interaction energy 
   by $E_N({\cal M},{\cal L})= \langle{\bf z}_{\cal L},{\bf p}_{\cal L}  \rangle$,
   see (\ref{eqn:EnergySum_ligand}) and Lemma \ref{lem:energy_ligand}.
   }
      \end{enumerate}
    Analysis of {\bf Algorithm PLIE} leads to the following conclusion.
 \begin{remark}\label{rem:TEIE}     
Once the long-range electrostatic energy is computed in rank-$R$ CP format, the calculation of the interaction energy of the protein and ligand at the given location of ${\cal L}$
is estimated by $O(R\, N_L)$, i.e. it is very cheap and practically does not depend on the size of the protein
and on the size of the 3D $n\times n \times n$ Cartesian representation grid.
\end{remark}
The algorithms TEP and PLIE described above  are based on the well established deterministic mathematical methods
such that given input data, the output can be calculated by finite sequence of linear/multilinear algebraic operations.

However, our algorithm for finding the appropriate posing positions via interaction energy minimization (PPIEM)
includes several heuristic procedures and does not guarantee in general the finding of all appropriate 
configurations realizing 
quasi optimal extreme points on PES. Therefore, in what follows, we describe only the main ingredients and 
principal beneficial features of the proposed techniques. Notice that presented tensor-based 
approach has several potential advantages, 
however their efficient numerical realization depends on the size of bio-molecule of interest thus 
requiring slightly modified versions of the computational algorithm. In this concern, we distinguish two versions
of the algorithm oriented either on \emph{moderate size} bio-molecule or onto large \emph{protein-complexes}. 

{\bf Algorithm PPIEM.} (Sketch).
Global minimization of the interaction energy.\\
(A0) \emph{Initialization:} define the angle resolution parameters
 $\delta_L=(2\pi)/m_L$ and $\delta_M=(2\pi)/m_P$, where $m_L$ and $m_P$ are 
the number of rotations for the ligand itself and rotations of the ligand around bio-molecule, respectively.\\
(A1) Given bio-molecule ${\cal M}$, calculate the rank-$R$ tensor-based representation of 
the long-range 
part of the electrostatic potential of ${\cal M}$ in the bounding box $\Omega_h$ by using Algorithm TEP. 
Application of Algorithm TEP is considered as the pre-computing step and performs only once.\\
(A2) Identify the ``geometric'' centers of the bio-molecule and ligand, denoted by $x_P$ and $x_L$, respectively.\\
(A3) Move the ligand from the initial position toward the center $x_P$ along the line from 
$x_L$ to $x_P$ to achieve the minimal admissible distance to the protein, given 
the van der Waals distance $\sigma >0$.\\ 
(A4) Calculate the energy for all rotations of the ligand around $x_L$ and store the corresponding data array.\\
(A5) Perform all rotations of the center of ligand $x_L$ around the center of protein, $x_P$, and for each 
position repeat calculations in item (A4). Store the results of all energy calculations.\\
(A6) Find all minimum points in the resultant data array composed of all local energy minimums
obtained in item (A4).\\
(A7) Return the ligand into the position of detected minimum(s), with the respective orientation, see item (A6),
thus representing the \emph{output for the appropriate posing positions}.\\

The main beneficial feature of the Algorithm PPIEM is the super-fast calculation of the interaction energy 
for every protein-ligand configuration at the cost $O(R\, N_L)$, 
 where $R$ is the optimized tensor rank of the 
long-range part in the electrostatic potential of large bio-molecule ${\cal M}$.
This cost weakly (logarithmically) depends on the size of the protein,
that allows to perform very fast the energy calculation (practically at the cost that is almost independent of
the size of protein) for each mutual protein-ligand location 
arising in the curse of energy optimization. 

\section{Numerical examples for moderate size bio-molecules} \label{sec:numerics}

\subsection{Numerics for synthetic data} \label{ssec:synthetic_example}

The above defined computational schemes  have been implemented in Matlab.
The numerical flowchart includes three main algorithms. 
  \begin{itemize}
   \item 
 The discretized collective electrostatic potential of the large biomolecule 
  is computed in the low-rank RS tensor format in the bounding box $[-B,B]^3$. 
  This is implemented by first generating the reference Newton kernel in a form of 
  a low-rank canonical tensor on a 3D  $n\times n\times n $ Cartesian grid. 
  The long-range part of the collective potential is then calculated by shifting and summing 
  the long-range vectors of the reference Newton kernel in correspondence with coordinates 
  of particles in the protein.
  The last step  is the canonical-to-Tucker rank reduction by RHOSVD. 
  Given protein molecule, the rank-$R$ tensor structured long range electrostatic potential 
  is pre-computed in parametrized form only once requiring $3 R n \ll n^3$ storage size. 
  \item Calculation of the electrostatic protein-ligand interaction energy (electrostatic 
  part of the free binding energy) for the given position of the ligand. This is performed by  weighted  
  sampling (with the weights equal to atomic charges in ligand) of the values of the  collective electrostatic 
  potential for the protein evaluated at the coordinates of particles in the ligand. 
  The numerical cost is estimated by $O(R \, N_L)$.
  \item In our deterministic scheme we place both protein and ligand in 
  the computational box. Then we move the ligand towards  the admissible 
  position(s) under the control of the van der Waals distance, and compute and store the binding energy for each 
  rotation of the ligand around its center. 
  The procedure is repeated for all prescribed  rotations around the protein. 
  Finally, we find the minimum(s) over the computed values of energy on the potential energy surface.  
   \end{itemize}

In the following numerical examples we demonstrate how the presented complex of 
tensor-based algorithms, TEP-PLIE-PPIEM, works on some synthetic data for  the molecular systems of atoms 
placed on the plane. 
In this case all forces act in the same plane and the whole dynamical trajectory of the ligand 
along that plane  is parametrized by three 
parameters, two azimuthal  rotation angles around the geometric centers of each bio-molecule and the 
radial distance between molecules.
\begin{figure}[htb]
\centering
\includegraphics[width=5.0cm]{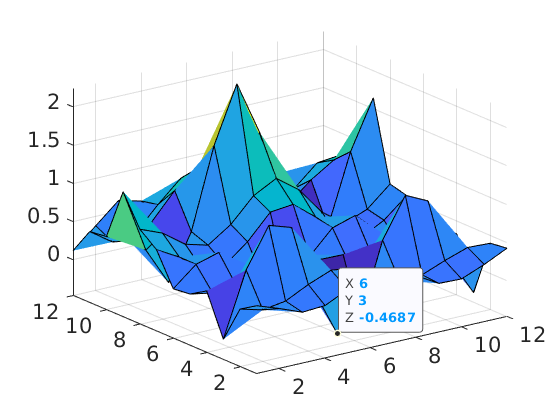}
\includegraphics[width=5.0cm]{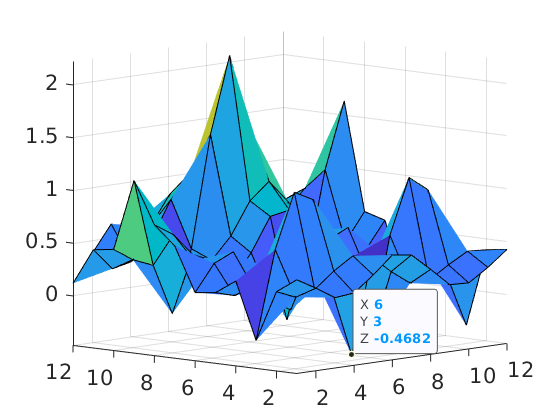}
\includegraphics[width=5.0cm]{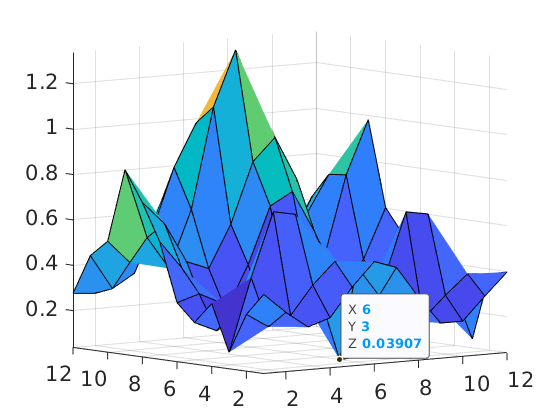}
\caption{\small Example of two-parametric PES with respect to two azimuthal  rotation angles 
(and corresponding translations) of both bio-molecule $\cal M$ and ligand $\cal L$ 
interacting in the docking process based on RS tensor decompositions of potential. 
Computations are performed on $n\times n\times n$ 
3D Cartesian grids with $n=512$ and with ranks $R_l=14, 10$ and $R_l= 8$ 
(figures left, middle and right, respectively). The CP rank of the reference tensor ${\bf P}_R $ is  $R=20$.  }
\label{fig:System_Particl_tens}
\end{figure}

In numerical practice we use the reduced parametric space so that the distance parameter varies only in the 
layer-like enveloping ring around the protein that reduces considerably the amount of computational job.
Hence, the trajectory of ligand $\cal L$ belongs to the narrow strip surrounding the bio-molecule 
$\cal M$ so that the atoms in $\cal L$ are separated from $\cal M$ via controlling the so-called van der Waals 
distance $\sigma >0$,
see for example the schematic illustration in Figure \ref{fig:layer_ind}.

Figure \ref{fig:System_Particl_tens} 
shows the examples of two-parametric PES with respect to two azimuthal  rotation angles 
 of both bio-molecule $\cal M$ and ligand $\cal L$ around their centers, 
 each cluster is analyzed in $12$ angular positions with the angular step-size $2\pi/12$.
 In each rotational position around $\cal M$ the local rotations of ligand 
 are considered on the minimal admissible distance (van der Waals distance) from protein $\cal M$.
 Here the results with different ranks $R_l=14, 10, 8$ of the 
long-range part in the reference tensor ${\bf P}_R$ (discretized single Newton kernel) are presented,
 where $R=20$ is the full rank of ${\bf P}_R$. Notice that calculations with $R_l <8$ do not recover 
 the correct posing configuration.
  \begin{figure}[htb]
\centering
\includegraphics[width=5.0cm]{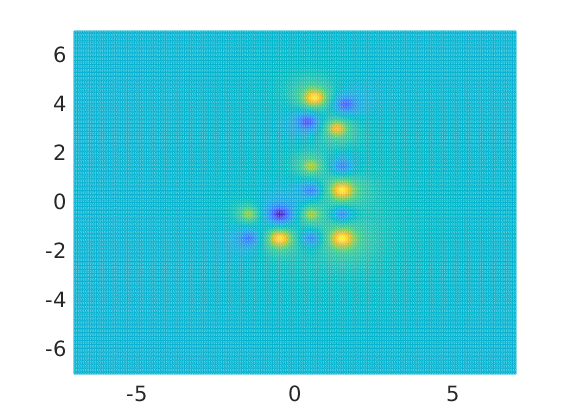}
\includegraphics[width=5.0cm]{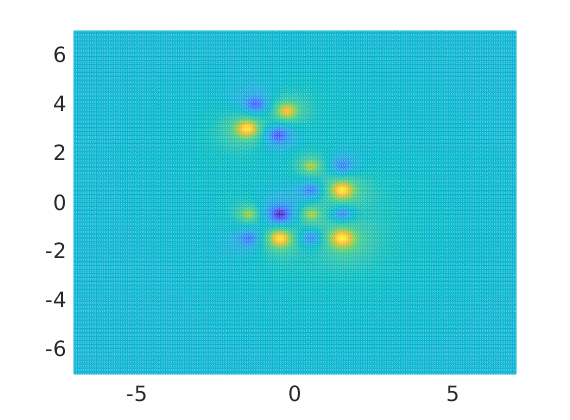}
\includegraphics[width=5.0cm]{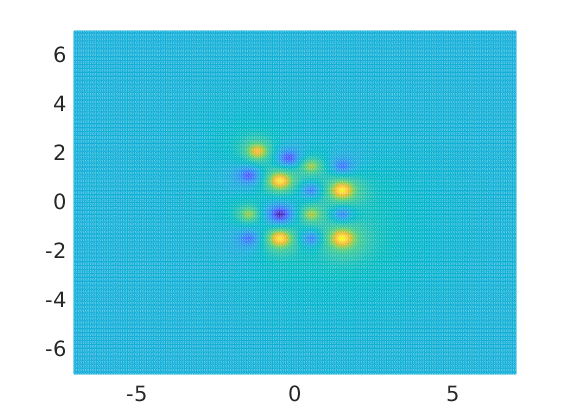}
\includegraphics[width=5.0cm]{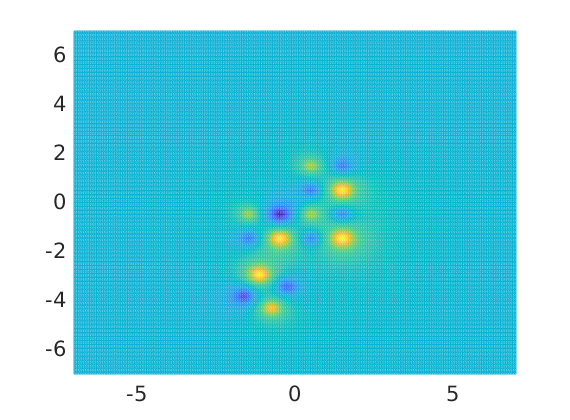}
\includegraphics[width=5.0cm]{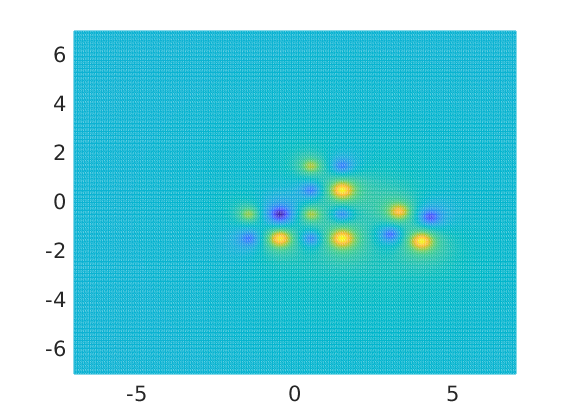}
\includegraphics[width=5.0cm]{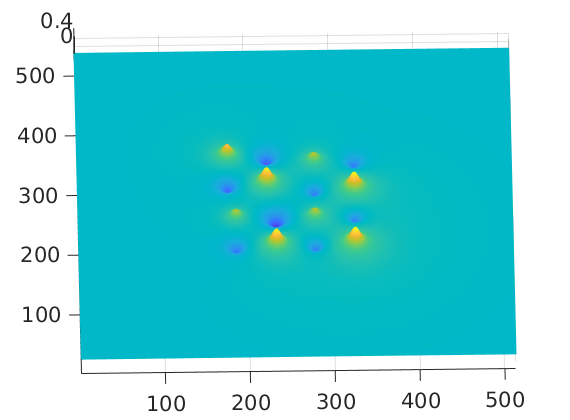}
\caption{\small Snapshots of the electrostatic potential of molecular clusters $\cal M$ and  $\cal L$ at
five different angles of azimuthal  rotations (from total $12\times 12$ mutual rotations). 
Final result of posing process is presented at bottom, right. }
\label{fig:System_Particl_tens_dynamics}
\end{figure}   
 
The numerics illustrate that energy minimization based on the low-rank parametrization of the long-range part
for the molecular clusters $\cal M$ finds the same posing positions 
both in the case of rather small rank parameters and
 in the case of full rank ($R=20$ for the reference tensor ${\bf P}_R$) tensor representation 
 of the electrostatic potential ${\bf P}_{\cal M}$.
The values of minimal energy (binding energy) remains almost the same in all three cases.

Figure \ref{fig:System_Particl_tens_dynamics} demonstrates five snapshots of the electrostatic potential 
of the synthetic ``protein-ligand'' construction   
of two molecular clusters $\cal M$ and  $\cal L$ (located in one plane) along the line of energy minimization
process for the docking positions, 
at different angles of azimuthal  rotations of both ``protein'' and ``ligand''. 
At every rotation angle the electrostatic interaction energy is computed by 
(\ref{eqn:EnergySum_partial_lig}), using the long-range tensor
representation of the electrostatic potential of the cluster $\cal M$.

\subsection{``Finding back'' a site for a detached fragment of a protein}.
\label{ssec:find_back_example}

In what follows, we describe the numerical examples of finding back the initial position 
of a fragment of the protein (a group of connected particles), after it was detached from some particular
place at the \emph{reference protein} with 379 atoms, see Fig. \ref{fig:Protein_example_1}, left. 

In example 1, we select a \emph{sample protein} composed of 79 atoms that 
is a part of the initial reference bio-molecule, see Fig. \ref{fig:Protein_example_1} (right). 
Then, we define a small ``ligand'', where the  latter is obtained 
by separating a small fragment (containing 5 atoms) of the sample protein (sub-protein) - a group of connected particles
at the protein surface, see Fig. \ref{fig:Protein_sample_1}, middle.  
The size and the location of the ``ligand'' on the protein is chosen rather  arbitrarily, 
with the only criterion that the   binding energy of the detached compound 
should be negative (the structure is stable).
  \begin{figure}[tbh]
\centering
\vspace{-2mm}
\includegraphics[width=5.4cm]{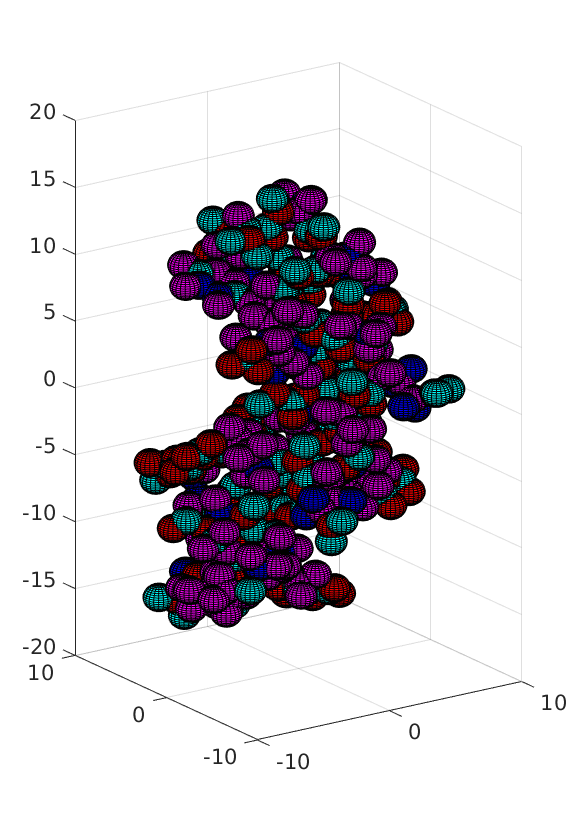}
\vspace{-5mm}
\caption{\small Example of the 379-atom protein.}
\label{fig:Protein_example_1}
\end{figure}

After detaching a group of particles, assigned as a ligand,  we then try to perform   ``docking'' by 
finding the original place of this ``ligand'' via searching for the minimum on the electrostatic interaction energy 
landscape. In our methodology, the search is done in a ``blind way'', without any preliminary knowledge 
of the original position of the detached group of particles. Interaction  energy of the protein is computed 
by using  the RS tensor representation of the electrostatic potential as shown in Lemma \ref{lem:energy_ligand}.
 
  In these numerical experiments we confine ourselves by a conformational search only at a given altitude level of the protein,
 that is on a given plane, by computing the  electrostatic interaction energy for azimuthal
 rotations of the ``ligand'' around the protein. Recall that we simplify the presentation, by considering the experiment  
 only using a part of an initial protein - sub-protein extracted in some layer of certain width along z-axis. 
 Our numerical experiments present the results of application of the  Algorithm PPIEM to the protein-ligand
 system described above.

 \begin{figure}[h]
\centering
\vspace{-0.5cm}
\includegraphics[width=7.4cm]{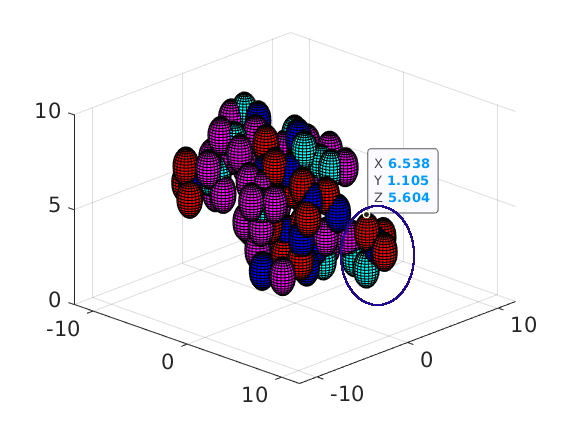}
\includegraphics[width=7.4cm]{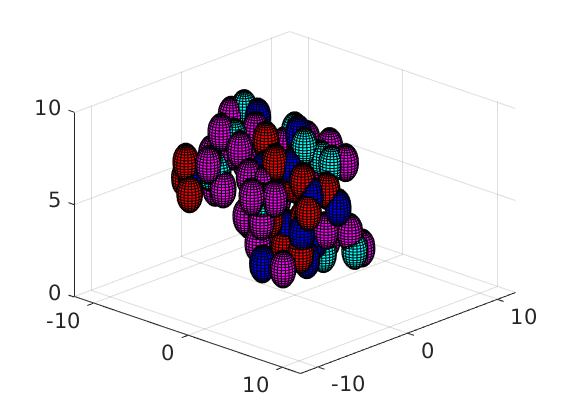}\hspace{-0.6cm}\\
\includegraphics[width=7.4cm]{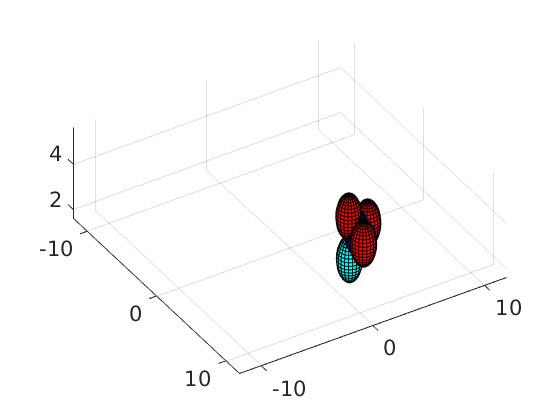}\hspace{-0.6cm}
\includegraphics[width=7.4cm]{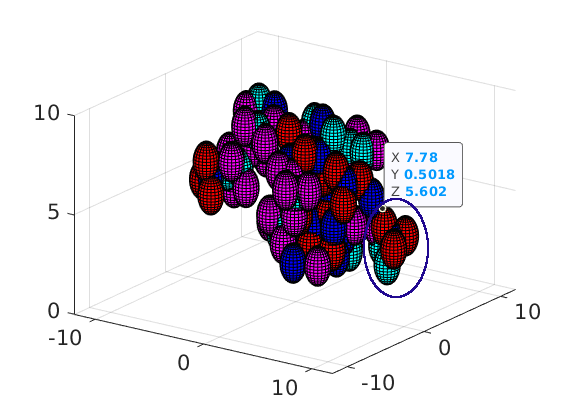}
\caption{\small Example 1. Top, left: the circle shows the group of atoms to be separated. The sub-problem without an extracted ``ligand'' (top, right),
the extracted ``ligand'' (bottom, left) and the result of docking with the extracted part (bottom, right). }
\label{fig:Protein_sample_1}
\end{figure}

 Figures \ref{fig:Protein_example_1} -- \ref{fig:Protein_sample_1} 
demonstrate a search of the position of the detached group of particles in a part of a protein molecule containing 379 atoms.  
The electrostatic interaction energy of the protein is $- 27.9488$ hartree.
Figure  \ref{fig:Protein_example_1}, left, shows the original protein. 
The following colors are used for the particles: red and blue colors denote the particles with positive and negative 
partial\footnote{Particle charges in protein data bases, called partial 
charges, are obtained by experimental measurements.} charges greater than 0.2 and lower than -0.2, correspondingly. 
For the particles with the partial charges in the interval $[-0.2,0.2]$, the magenta and cyan colors are used. 

Figure \ref{fig:Protein_example_1}, right, shows the part of the reference protein (sample protein) 
selected in the z-axis interval [2.5, 9] \AA{}, 
where we perform a detachment of a group of particles.   The location of this group is pointed by the data tip, 
showing the (x,y,z)-coordinates of one of the particles, that is $(6.5,1.1,5.6)$ \AA{}. 
The electrostatic interaction energy of this part of the protein is $- 3.45 $ 
hartree. Figure \ref{fig:Protein_sample_1}, left, shows the part of the 
protein, where this group of particles is missing (detached).
The energy of this group becomes higher, $- 3.28 $ hartree.
Figure \ref{fig:Protein_sample_1}, middle, shows the detached ``ligand'' (with the energy $- 0.18 $ hartree), 
for which we should find back the matching position. 
  Figure \ref{fig:Protein_sample_1}, right, shows the result 
 of the search for conformation site by using our algorithms. 
 It is shown, that using our blind search, the ligand has found its position only slightly different from the initial one. 
 The (x,y,z) coordinate of the same particle as in Figure \ref{fig:Protein_example_1}, right, is now $(7,7,0.5,5.6)$\AA{},
 see Fig. \ref{fig:Protein_sample_1}, right. 
 The electrostatic interaction energy of the sub-protein with docked ``ligand'' is $- 3.50 $ hartree, which is even lower, 
   than that of the initial configuration. 
   
   Figure \ref{fig:Potential_example_1} visualizes the electrostatic potential 
   for the initial (left) and final (right) configurations of the docked sub-protein. 
    Figure \ref{fig:Protein_matrix} (left) shows the surfaces of the electrostatic interaction 
   energies for sub-protein of example 1 computed for every rotation angle of ligand, i.e.
   24 rotations around the sub-protein, and 12 rotations of the ligand around its center.
   
   \begin{figure}[tbh]
\centering
\includegraphics[width=6.8cm]{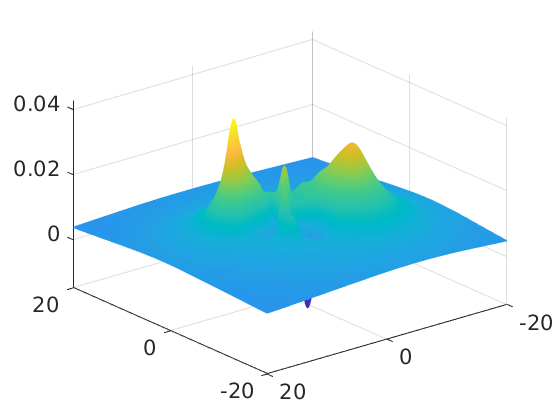}
 \includegraphics[width=6.8cm]{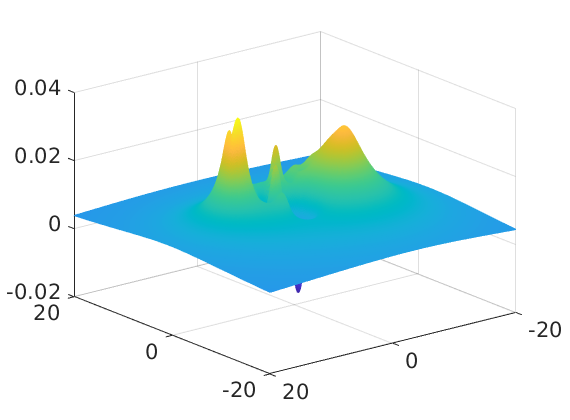}
\caption{\small The electrostatic potential of the initial sub-protein in example 1 at the search plane (left) 
and of the final result of docking process (right).}
\label{fig:Potential_example_1}
\end{figure}

In example 2, we present another simulation results for the small part of the protein in 
the z-axis interval [8, 15] \AA{} are presented    in Figure \ref{fig:Protein_example_3}. 
Electrostatic interaction energy of this part of the protein is $-4.79$ hartree, while its 
  energy after detaching a small group of particles (bottom left), increases to $-4.71 $ hartree. 
  After finding the 
right docking position of the detached group of particles (``ligand'') the resulting energy  
becomes $-4.87$ hartree. The respective potential energy surface corresponding to 28 rotations 
around the sub-protein and 12 self-rotations of the 
``ligand'' is presented in Figure \ref{fig:Protein_matrix} (right).

 \begin{figure}[tbh]
\centering
\includegraphics[width=7.6cm]{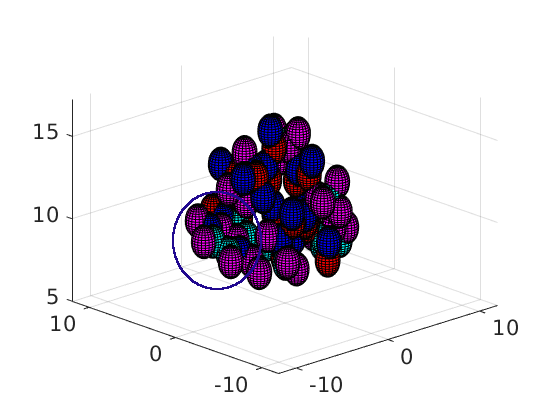}
\includegraphics[width=7.6cm]{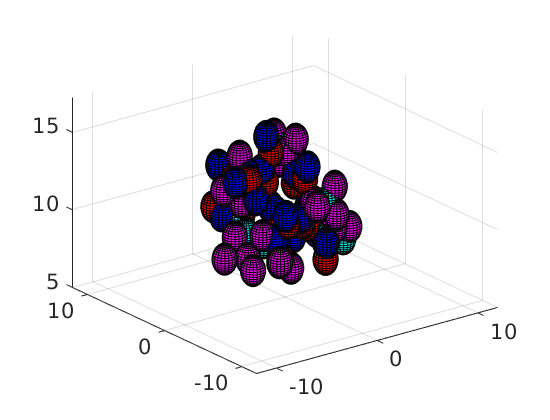}
\vspace{-0.3cm}
\includegraphics[width=6.0cm]{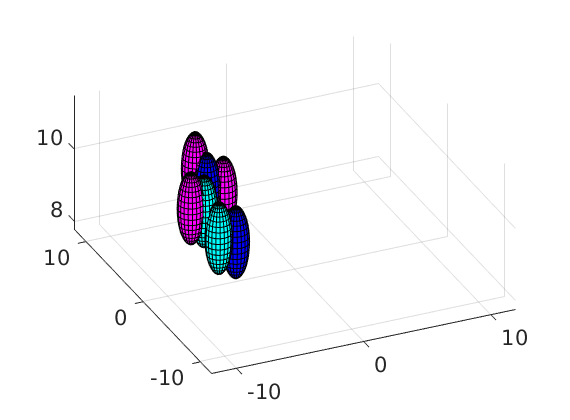}
\includegraphics[width=7.6cm]{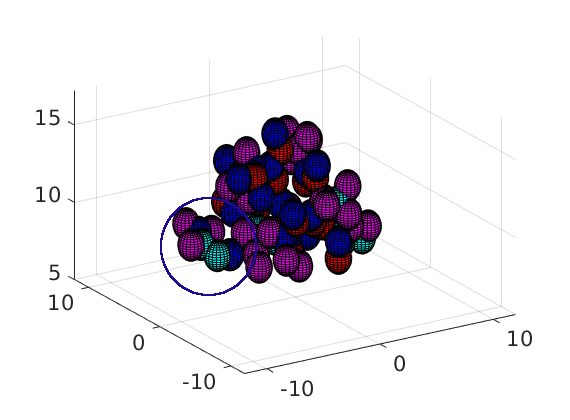}
\caption{\small Sub-protein example 2. Top: initial part of the protein (left), and the same part 
without  an extracted ``ligand'' (right).
Bottom: The extracted ``ligand'' (left) and the result of docking of the extracted part (right).  }
\label{fig:Protein_example_3}
\end{figure}

\begin{figure}[htb]
\centering
\includegraphics[width=7.3cm]{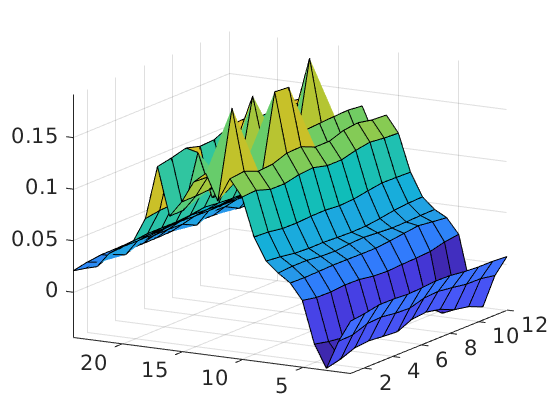}
\includegraphics[width=7.3cm]{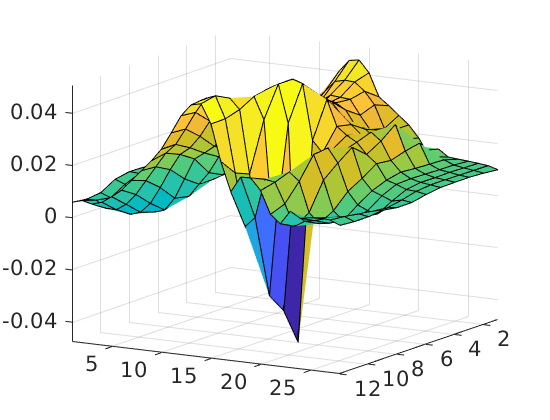}
\caption{\small The electrostatic potential energy surface
for the above examples 1 and 2, shown in Figures \ref{fig:Protein_sample_1} (left) and 
\ref{fig:Protein_example_3} (top left).}
\label{fig:Protein_matrix}
\end{figure}

 
 \section{Looking to further prospects}\label{sec:Docking_Prospects_summary} 
 
\subsection{Towards model reduction: why and how}
\label{ssec:model_reduct} 
 The benefits of numerical techniques presented in the paper are mainly concerned with the use of
substantial data compression via data-sparse rank-structured tensor parametrization of the electrostatic potential 
 and some other simplifications of the original problem setting thus leading to the model reduction, see also \cite{BenFa-book_2024} for theoretical background.
 
 We summarize that our approach aims on the substantial reduction of the computational and storage costs of the numerical schemes which is based on the following tools to be implemented by 
using low rank tensor representations of the interaction potential and respective multi-linear algebra.
 \begin{itemize}
  \item The low $\varepsilon$-rank representation of the collective electrostatic potential of the  bio-molecule in the RS tensor format.
 \item Fast evaluation of the corresponding   energy functional and force field  by using   tensor representation of the potential in the full computational box.
  \item Optimization of the steepest descent search on the potential energy surface only in the narrow enveloping strip around the bio-molecule.
  \item Using add-and-compress strategy for the robust rank reduction procedure in case of large molecular systems.
  \end{itemize}
Our techniques particularly apply to the system of one single bio-molecule (say, protein) and small ligand compound.
The constrained free-space energy minimization benefits from enhancing the energy calculation by using only 
the low-rank representation of the long-range part in the electrostatic potential of large bio-molecule.

The presented approach could be easily modified to the case of molecular complexes composed of several large bio-molecules
${\cal M}_1,\ldots,{\cal M}_P$.
In this case the energy expression (\ref{eqn:EnergySum_partial_lig}) should be modified as follows (say, on molecule 
${\cal M}_P$ considered as a ligand),
\begin{equation}\label{eqn:EnergySum_Mcomplex}
E_N({\cal M}_1,\ldots,{\cal M}_P)= \sum\limits_{{\ell}_P=1}^{M_P} z_{{\ell}_p} 
\left(\sum\limits_{p=1}^{P-1} P_{{\cal M}_p}\right)(x_{\ell_p}),\quad x_{\ell_p} \in {\cal M}_P,
\end{equation}
 where $P_{{\cal M}_p}$ is the long-range part of the electrostatic potentials of the bio-molecule ${\cal M}_p$,
$p=1,\ldots,P$, in the respective exterior domain.

Another aspect is related to the more refined considerations based on the electrostatics calculation 
in polarized media by using the Poisson-Boltzmann equation (PBE) or some of its simplifications.
This issue will be addressed elsewhere. Here we only notice that rather accurate approximation of the far-field potential
can be calculated by using the analytic solution for the spherical solute region. 
For example, this can be used to determine the approximate boundary conditions for PBE posed in bounded domain.

In concern with the docking problem we notice that
the discretization of PBE on large $n\times n \times n$ grid in the bounding box $[-B,B]^3$ 
leads to the solution of huge system of linear (nonlinear)
equations with the matrix size $n^3 \times n^3$, see \cite{BeKKKS:21,KKKSBen:21,KwFSBen:22}. 
This indicates possible limitations of the PBE equation 
for numerical modeling of electrostatic in large bio-molecular systems. 
Moreover, the dynamical minimization 
of the binding energy functional in the protein-ligand systems requires solution of this equation for large 
amount of different configurations of the equation coefficients (different positions of ligand)
in the course of multi-parametric energy minimization. 
This makes the PBE model not promising in the case of molecular complexes and even 
in case of rather large protein-ligand systems.

In this concern, the practically tractable mathematical techniques for finding the appropriate binding sites 
requires certain model reduction for PBE that properly recovers the interaction potential in the region
outside of the protein. It is also interesting to understand how to adapt the method developed in this paper
(for the case of free space molecular system) to the solution of posing problem in the case of polarized media.
These questions will be addressed elsewhere. 

\subsection{Collective force field for rigid clusters of charged particles}
\label{ssec:Forces-Cluster}

In numerical techniques based on deterministic/stochastic dynamics the efficient gradient 
(or stochastic gradient) calculation plays the central role.
In the framework of docking problems one deals with the dynamics of rigid clusters of charged particles.
In this setting the force calculation requires certain modification compared with the case 
of single particles, see \S\ref{ssec:Forces_Applic}.

One of the specific problem of rigid docking is that the particles in ligand could not move independently 
since the geometric shape 
of the cluster ${\cal L}$ is fixed and the atoms in ligand do not have individual degrees of freedom
($d$ parameters per particle).
This means that their collective dynamics is uniquely determined by only few parameters, 
the position $x_0$ ($d$ parameters)
of the geometric center and by the rotational angles of the system around this center, i.e., 
by $d-1$ angular parameters. Hence the total number of parameters, $2d-1$, does not depend on $L$, the size of 
ligand ${\cal L}$. 

Given position of the bio-molecule ${\cal M}$, we are interested in calculation of force vectors applied 
only to the whole rigid system
${\cal L}$, thus predicting the motion of geometric center $x_0$, described by the force vector ${\bf F}_0(x_0)$, 
as well as by the variation of rotational angles around the fixed point $x_0$, 
described the set of vectors $\{{\bf F}_{0,{\phi}}\}$, applied to some fixed point $x_\phi$. This set includes $d-1$ vectors,
one for the case of flat molecules (one angular parameter), and two vectors in case of full 3D docking process
(two angular parameters). 

The vectors ${\bf F}_0(x_0)$ and ${\bf F}_{0,{\phi}}$ can be calculated as the resultant of forces 
${\bf F}_\ell$, $\ell=1,\ldots,L$, applied to each of $L$ particles in the rigid cluster ${\cal L}$.
The total numerical cost of the resultant force calculation will 
be estimated based on Lemma \ref{lem:energy_ligand}, that is $(d-1)RL$ as the dominating term, plus the cost of simple 
geometric operations with $L$ force vectors calculated at each particle center in ligand.

Let the unit vectors ${\bf v}_\ell$ be directed from the cluster center $x_0$ to the cluster particles located at 
points $x_\ell \in {\cal L}$, $\ell=1,\ldots,L$, and ${\bf n}_\ell$ be the normal vectors to ${\bf v}_\ell$
in the plane spanned by the pare $[{\bf v}_\ell, {\bf F}_\ell]$ and directed from point $x_\ell$. 
Denote by ${\bf F}_{x_0,x_\ell}$ and ${\bf F}_{\phi,x_\ell}$  the projections of ${\bf F}_\ell$
onto the vectors ${\bf v}_\ell$ and ${\bf n}_\ell$, respectively, such that 
${\bf F}_\ell= {\bf F}_{x_0,x_\ell} + {\bf F}_{\phi,x_\ell}$.
Shifting all vectors ${\bf F}_{x_0,x_\ell} \mapsto {\bf F}_{x_0,\ell}$ to the chosen geometric 
center $x_0$ and summing them up we obtain
\begin{equation}\label{eqn:Field_Xcenter}
 {\bf F}_0(x_0)=\sum\limits_{\ell=1}^L {\bf F}_{x_0,\ell}.
\end{equation}
To calculate the tangential force, we chose the reference circle (sphere in 3D) of radius $r_0$ and shift all vectors
${\bf F}_{\phi,x_\ell} \mapsto {\bf F}_{\phi,\ell}$ along direction of  ${\bf v}_\ell$ to the distance $r_0$ 
from $x_\ast$ with the proper scaling
and preserving their direction along the normal vector ${\bf n}_\ell$. 
Now all adduced tangential vectors ${\bf F}_{\phi,\ell}$
act along the same circular (spherical) curve and hence their total action is calculated by the 
simple summation in tangential direction, to obtain
\begin{equation}\label{eqn:Field_phi}
 {\bf F}_{0,{\phi}}= \sum\limits_{\ell=1}^L {\bf F}_{\phi,\ell}.
\end{equation}
Now we notice that the cost of force calculation at the fixed point, ${\bf F}_\ell$, is of the order of $O(d(d-1) R_L)$,
i.e. of the same order as those for the potential evaluation at the fixed atomic center.
This proves the following lemma.
\begin{lemma}\label{lem:Force_ligand}
Given force vectors ${\bf F}_\ell$, $\ell=1,\ldots,L$, on each of $L$ particles in ${\cal L}$,
 the translation and rotational parts of collective forces acting on the rigid ligand ${\cal L}$ can be calculated 
 by (\ref{eqn:Field_Xcenter}) and (\ref{eqn:Field_phi}), respectively.
 The numerical cost for the force calculation on the rigid ligand ${\cal L}$ is estimated by 
 \[
  Q({\bf F}_0(x_\ast),{\bf F}_{0,{\phi}}) = O(R_L L),
 \]
 where $R_L$ is the CP rank of the long-range part in the trace of tensor ${\bf P}_{\cal M}$
 over small bounding box $\Pi \supset {\cal L}$, ${{\bf P}_{{\cal M},R}^{long}}_{|\Pi}$.
\end{lemma}
Lemma \ref{lem:Force_ligand} indicates that the cost of force calculation for rigid ligand is of the same 
order as for the binding energy calculation in some fixed position of the ligand ${\cal L}$.
Hence it can be useful in combination with the detailed investigation of the local energy landscape.

 \section{Conclusions}\label{sec:Docking_Conclusions}

 We introduce the new numerical method for fast calculation of electrostatic interaction energy of large multi-particle systems which explicitly use the 
 low-rank representation for the long-range part in electrostatic potential of a bio-molecule.
  In application to numerical modeling of protein-ligand docking problem, this allows fast energy and force computation
 where the constrained energy minimization process includes multiple evaluation 
 of the energy functional in the course of posing process. 
 
Lemma \ref{lem:energy_ligand} leads to the important conclusion that application of the RS tensor decomposition 
of multi-particle electrostatic potential 
allows an implementation of the posing algorithm where  the numerical cost of the binding energy calculation
very mildly (logarithmically) depends on the size of the target bio-molecule  that may include many hundred or even thousand atoms. 
Hence, this cost is practically proportional to the number of dynamical 
steps in the energy optimization algorithm, as well as to the size of a small ligand molecule.

 Introducing this beneficial numerical techniques is one of the main results of this paper.
We demonstrate the practical performance of our method on some synthetic examples as well for moderate size bio-molecules. 
In what following, we underline the benefits of the presented method:
\begin{itemize}
 \item Efficient tensor-based calculation of the collective electrostatic potential of the protein-ligand complex on the 
   $n\times n \times n$ 3D grid and its compact $O(n)$-complexity parametric  representation by using the RS tensor format.
   \item Fast  energy and force calculation  at the cost that is almost independent on the size of protein.
  \item Tensor numerical methods allow to use large $n\times n \times n$ 3D representation grids of size of the order of $n^3 \approx 10^{15}$
  which makes it possible to accurately calculate the electrostatic 
  potentials and interaction energy of large protein complexes  (or even parts of the DNA).
 \end{itemize}
 
Presented theoretical arguments and respective numerical analysis 
justify  \emph{the proof of concept} for possible
application of tensor-based computational techniques in the framework of real-life problems such as 
numerical modeling of protein-ligand docking and classification of large bio-molecules as well as in the deterministic/stochastic multi-particle dynamics.

\begin{footnotesize}

 \bibliographystyle{unsrt} 
\bibliography{BSE_Fock_Sums1.bib,docking_1.bib}
\end{footnotesize}

\end{document}